\documentclass[11pt,reqno]{amsart}

\usepackage{xcolor}
\usepackage[T1]{fontenc}
\usepackage{amsmath}							
\usepackage{amssymb}
\usepackage{amsthm}
\usepackage{amscd}
\usepackage{amsfonts}
\usepackage{mathabx}
\usepackage{stmaryrd}
\usepackage[all]{xy}

\usepackage{caption}
\usepackage{subcaption}
\usepackage{euler}

\usepackage{extarrows}

\usepackage[colorlinks, linktocpage, citecolor = purple, linkcolor = blue]{hyperref}
\usepackage{color}

\usepackage{tikz}									
\usetikzlibrary{matrix}
\usetikzlibrary{patterns}
\usetikzlibrary{positioning}
\usetikzlibrary{decorations.pathmorphing}
\usetikzlibrary{cd}

\usepackage{ytableau}
\usepackage{fullpage}
\usepackage[shortlabels]{enumitem}

\newtheorem{theorem}{Theorem}[section]
\newtheorem{lemma}[theorem]{Lemma}
\newtheorem{proposition}[theorem]{Proposition}

\theoremstyle{definition}
								
\newtheorem{definition}[theorem]{Definition}
\newtheorem{example}[theorem]{Example}

\newtheorem*{remark*}{Remark}

\newcommand{\ZZ}{\mathbb{Z}}
\newcommand{\CC}{\mathbb{C}}

\newcommand{\RR}{\mathbb{R}}

\newcommand{\oh}{\overline{h}}
\newcommand{\un}{\underline{n}}

\newcommand{\te}{\widetilde{e}}
\newcommand{\tf}{\widetilde{f}}

\newcommand{\tv}{\widetilde{v}}

\newcommand{\tX}{\widetilde{X}}

\renewcommand{\oe}{\overline{e}}

\DeclareMathOperator{\Ker}{Ker}
\DeclareMathOperator{\Stab}{Stab}

\DeclareMathOperator{\Vol}{Vol}

\let\Im\relax
\DeclareMathOperator{\Im}{Im}

\newcommand{\calB}{\mathcal{B}}

\newcommand{\calI}{\mathcal{I}}
\newcommand{\calM}{\mathcal{M}}

\DeclareMathOperator{\Aut}{Aut}

\DeclareMathOperator{\val}{val}

\DeclareMathOperator{\Div}{Div}

\DeclareMathOperator{\Jac}{Jac}

\title{Critical groups in harmonic abelian quotients}

\author{Mariia Vetluzhskikh}
\address{University of Maryland, College Park, MD 20742}
\email{S4pe2eAud3@protonmail.com}
\author{Dmitry Zakharov}
\address{Department of Mathematics, Central Michigan University, Mount Pleasant, MI 48859}
\email{\href{mailto:dvzakharov@gmail.com}{dvzakharov@gmail.com}}

\begin{document}

\begin{abstract}
    A harmonic cover of graphs $p:\widetilde{X}\to X$ induces a surjective pushforward morphism $p_*:\Jac(\tX)\to \Jac(X)$ on the critical groups. In the case when $p$ is Galois with abelian Galois group, we compute the order of the kernel of $p_*$, and hence the relationship between the numbers of spanning trees of $\widetilde{X}$ and $X$, in terms of Zaslavsky's bias matroid associated to the cover $p:\widetilde{X}\to X$.
\end{abstract}

\maketitle

\section{Introduction}

We investigate the relationship between three algebraic objects associated to graphs and their Galois covers. We may view a finite connected graph $X$ as a discrete analogue of an algebraic curve, and define divisor theory on $X$ in terms of a chip-firing operation (see~\cite{2007BakerNorine},~\cite{2018CorryPerkinson}). The discrete analogue of the Jacobian variety of a curve is the \emph{Jacobian group} $\Jac(X)$, also known as the \emph{critical group} of $X$ in the combinatorics literature. It is a finite abelian group whose order is equal to the number of spanning trees of $X$. On the other hand, we may view $X$ as a combinatorial version of a number field, and define primes as certain equivalence classes of closed paths on $X$. By analogy with the Dedekind zeta function of a number field, we define the \emph{Ihara zeta function} $\zeta(s,X)$ (see~\cite{1992Bass}), and $\zeta(s,X)$ determines $|\Jac(X)|$ by an appropriate version of the class number formula~\cite{1998Northshield}. Finally, we consider the \emph{graphic matroid} $\calM(X)$ on the edge set of $X$. The bases of $\calM(X)$ are the spanning trees, so that $|\Jac(X)|$ is the number of bases of $\calM(X)$.

The discrete analogue of a finite map of algebraic curves is a \emph{harmonic morphism} of graphs $p:\widetilde{X}\to X$ (introduced in~\cite{2000Urakawa}), which is required to satisfy a balancing condition~\eqref{eq:localbalancing} at the vertices of $\tX$ in terms of a local degree function $d_p:V(\tX)\to \ZZ_{\geq 1}$. Among other things, a harmonic morphism $p:\widetilde{X}\to X$ defines a surjective pushforward map $p_*:\Jac(\tX)\to \Jac(X)$ on the Jacobian groups. A natural question is to determine the order $|\Ker p_*|=|\Jac(\tX)|/|\Jac(X)|$ of the kernel, or, alternatively, to compare the number of spanning trees of $\tX$ and $X$. Since $|\Jac(X)|$ can be computed using either the Ihara zeta function $\zeta(s,X)$ or the graphic matroid $\calM(X)$, we may expect that appropriate generalizations determine the order of the kernel.

This question is difficult to approach in full generality, and we restrict our attention to two distinguished classes of harmonic morphisms. First, let $G$ be a finite group. A harmonic morphism $p:\tX\to X$ is said to be a \emph{Galois cover} with Galois group $G$ if $X$ is the quotient of $\tX$ by a $G$-action that is free on the edges (see Definition~\ref{def:Gcover} for a precise statement). In our paper, we generally further restrict our attention to the case when $G$ is abelian. Second, we say that a harmonic morphism $p:\tX\to X$ is \emph{free} if all local degrees of $p$ are equal to 1. A free harmonic morphism of graphs is the same as a covering space in the topological sense, and free Galois covers are those that are defined by free group actions. 

Let $p:\tX\to X$ be a free Galois cover with Galois group $G$, not necessarily abelian. In~\cite{1996StarkTerras} and~\cite{2000StarkTerras}, Stark and Terras defined the \emph{Artin--Ihara $L$-functions} $L(s,\tX/X,\rho)$ of the cover, indexed by representations $\rho$ of $G$, and showed that the Ihara zeta function $\zeta(\tX,s)$ factors as a product of the $L(s,\tX/X,\rho)$ over the irreducible representations. Since $\zeta(s,\tX)$ determines $|\Jac(\tX)|$, this gives a formula for $|\Ker p_*|$ in terms of the $L$-functions (see~\cite{2024HammerMattmanSandsVallieres}). On the other hand, in~\cite{2014ReinerTseng} Reiner and Tseng investigated the subgroup $\Ker p_*$ in detail, and found a matroidal formula for its order when the Galois group is $\ZZ/2\ZZ$ (this formula was rediscovered in the context of double covers of metric graphs in~\cite{2022LenZakharov}). In this case, the cover determines the structure of a \emph{signed graph} $\sigma:E(X)\to \{\pm 1\}$ on $X$, and the order of $\Ker p_*$ is a weighted sum over the bases of Zaslavsky's \emph{signed graphic matroid} $\calM(X,\sigma)$ (see Proposition 9.9 in~\cite{2014ReinerTseng}).

The main theorem of our paper is a matroidal formula for the order of $\Ker p_*$ in the case when $p:\tX\to X$ is an abelian Galois cover, not necessarily free. We employ the theory of Artin--Ihara $L$-functions for the proof, but the formula for $|\Ker p_*|$ is given in terms of weights of certain matroids $\calM(\tX/X,\rho)$ associated to the cover and a representation $\rho$ of the Galois group.

In order to state our results, we first recall how to construct abelian Galois covers. Free Galois covers of a fixed graph $X$ with abelian Galois group $G$ are described by elementary algebraic topology: any such cover $p:\tX\to X$ is given by an element $[\eta]\in H^1(X,G)$ of the simplicial cohomology group of $X$ with coefficients in $G$ (the element $\eta$ is called a \emph{$G$-voltage assignment} in the combinatorics literature). This description was generalized to harmonic abelian covers in~\cite{2024LenUlirschZakharov}. Given such a cover $p:\tX\to X$, the abelian Galois group $G$ acts on the fiber $p^{-1}(v)$ over each vertex $v\in V(X)$, and the stabilizer of any vertex in the fiber $p^{-1}(v)$ is a subgroup $D(v)\subseteq G$ depending only on $v$, called the \emph{dilation group} of $v$. The collection of dilation groups defines a \emph{$G$-dilation datum} $D=\{D(v):v\in V(X)\}$ on $X$, and isomorphism classes of harmonic covers with Galois group $G$ and fixed dilation datum $D$ are in bijection with a certain \emph{dilated cohomology group $H^1(X,D)$}.

On the other hand, Galois covers of graphs have a natural matroidal aspect. Given a free Galois cover $p:\tX\to X$ defined by a $G$-voltage assignment $[\eta]\in H^1(X,G)$, Zaslavsky defines in~\cite{1989Zaslavsky,1991Zaslavsky} a \emph{bias matroid} $\calM(X,\eta)$ on the set of edges of $X$ (for $G=\ZZ/2\ZZ$, this is Zaslavsky's signed graphic matroid of~\cite{1982Zaslavsky}). In this paper, we extend Zaslavsky's construction in two elementary ways. First, we reinterpret Zaslavsky's matroid geometrically in terms of the cover $p:\tX\to X$ and show how to define a matroid $\calM(\tX/X)$ on the edge set of $X$ when $p:\tX\to X$ is a harmonic abelian cover, not necessarily free. Second, we twist the definition of $\calM(\tX/X)$ by a character $\rho$ of $G$ to obtain a collection of matroids $\calM(\tX/X,\rho)$ indexed by the $|G|-1$ nontrivial characters of $G$. It is more convenient to work with the dual matroids $\calM^*(\tX/X,\rho)$, and we equip their bases with certain \emph{weights} (which are valued, generally speaking, in the real cyclotomic field). The \emph{weight} $w(\calM^*(\tX/X,\rho))$ of the matroid is the sum of the weights of the bases.

We are now ready to state our main result.

\begin{theorem}
    Let $p:\tX\to X$ be a harmonic Galois cover of graphs with abelian Galois group $G$ of order $N$. The number of spanning trees of $\tX$ is equal to 
    \begin{equation}
    |\Jac(\tX)|=\frac{1}{N}\prod_{v\in V(X)}|D(v)|^{N/|D(v)|}|\Jac(X)|\prod_{\rho\in \widehat{G}'}w(\calM^*(\tX/X,\rho)), 
    \end{equation}
    where $D(v)$ is the dilation group at the vertex $v\in V(X)$, $|\Jac(X)|$ is the number of spanning trees of $X$, $\widehat{G}'$ is the set of nontrivial characters of $G$, and $w(\calM^*(\tX/X,\rho))$ is the weight of the matroid $\calM^*(\tX/X,\rho)$.
    \label{thm:mainintro}
\end{theorem}

\begin{example} \label{ex:icosahedron1} We compute the number of spanning trees of the icosahedral graph $\tX$. The generator of the group $G=\ZZ/5\ZZ$ acts on $\tX$ by a rotation by $2\pi/5$ about an axis passing through two opposite vertices. The quotient graph $X$ is shown on Figure~\ref{fig:icosahedron} and has $|\Jac(X)|=2$ spanning trees. The dilation groups are $D(v_1)=D(v_4)=\ZZ/5\ZZ$ on the outer vertices of $X$ and are trivial on the inner vertices. The group $\ZZ/5\ZZ$ has four nontrivial characters
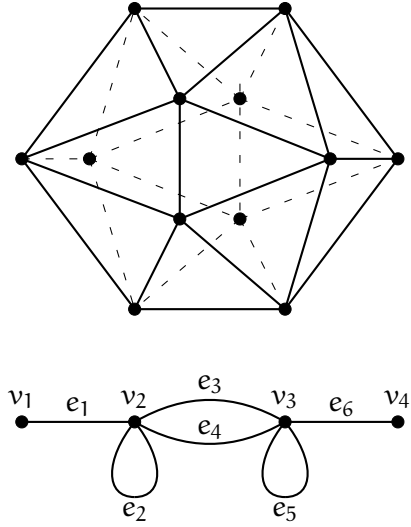
\begin{figure}
    \centering
    \begin{tikzpicture}

    \node at (4.5,0.3) {$v_1$};
    \node at (6,0.3) {$v_2$};
    \node at (8,0.3) {$v_3$};
    \node at (9.5,0.3) {$v_4$};

    \draw[fill](4.5,0) circle(.08);
    \draw[fill](6,0) circle(.08);
    \draw[fill](8,0) circle(.08);
    \draw[fill](9.5,0) circle(.08);

    \draw[thick] (4.5,0) -- (6,0);
    \draw[thick] (8,0) -- (9.5,0);
    \draw[bend left, thick] (6,0) to (8,0);
    \draw[bend right, thick] (6,0) to (8,0);
    \draw [thick] (6,0) .. controls (5.6,-0.5) and (5.6,-1) .. (6,-1);
    \draw [thick] (6,0) .. controls (6.4,-0.5) and (6.4,-1) .. (6,-1);
    \draw [thick] (8,0) .. controls (7.6,-0.5) and (7.6,-1) .. (8,-1);
    \draw [thick] (8,0) .. controls (8.4,-0.5) and (8.4,-1) .. (8,-1);
    
    \node at (5.25,0.2) {$e_1$};
    \node at (6,-1.2) {$e_2$};
    \node at (7,0.5) {$e_3$};
    \node at (7,-.1) {$e_4$};
    \node at (8,-1.2) {$e_5$};
    \node at (8.75,0.2) {$e_6$};

    \begin{scope}[shift={(0,3.5)}]

        \draw[fill](4.5,0) circle(.08);

        \draw[fill](6,2) circle(.08);
        \draw[fill](6,-2) circle(.08);
        \draw[fill](6.6,0.8) circle(.08);
        \draw[fill](6.6,-0.8) circle(.08);
        \draw[fill](5.4,0) circle(.08);
        
        \draw[fill](8,2) circle(.08);
        \draw[fill](8,-2) circle(.08);
        \draw[fill](7.4,0.8) circle(.08);
        \draw[fill](7.4,-0.8) circle(.08);
        \draw[fill](8.6,0) circle(.08);

        \draw[fill](9.5,0) circle(.08);
        
        \draw[thick] (4.5,0) -- (6.6,0.8);
        \draw[thick] (4.5,0) -- (6.6,-0.8);
        \draw[thick] (6.6,0.8) -- (6.6,-0.8);
        \draw[thick] (4.5,0) -- (6,2);
        \draw[thick] (4.5,0) -- (6,-2);
        \draw[thick] (6,2) -- (6.6,0.8);
        \draw[thick] (6,-2) -- (6.6,-0.8);

        \draw[thick] (8,2) -- (6,2);
        \draw[thick] (8,-2) -- (6,-2);
        \draw[thick] (6.6,0.8) -- (8.6,0);
        \draw[thick] (6.6,-0.8) -- (8.6,0);
        \draw[thick] (6.6,0.8) -- (8,2);
        \draw[thick] (6.6,-0.8) -- (8,-2);
        
        \draw[thick] (9.5,0) -- (8.6,0);
        \draw[thick] (8.6,0) -- (8,2);
        \draw[thick] (8,2) -- (9.5, 0);
        \draw[thick] (8,-2) -- (9.5, 0);
        \draw[thick] (8,-2) -- (8.6,0);
        
        \draw[loosely dashed] (4.5,0) -- (5.4,0);
        \draw[loosely dashed] (5.4,0) -- (6,2);
        \draw[loosely dashed] (5.4,0) -- (6,-2);
        \draw[loosely dashed] (5.4,0) -- (7.4,0.8);
        \draw[loosely dashed] (5.4,0) -- (7.4,-0.8);
        \draw[loosely dashed] (7.4,0.8) -- (6,2);
        \draw[loosely dashed] (7.4,0.8) -- (8,2);
        \draw[loosely dashed] (7.4,0.8) -- (9.5,0);
        \draw[loosely dashed] (7.4,-0.8) -- (9.5,0);
        \draw[loosely dashed] (7.4,1) -- (7.4,-0.8);
        \draw[loosely dashed] (7.4,-0.8) -- (6,-2);
        \draw[loosely dashed] (7.4,-0.8) -- (8,-2);

    \end{scope}
\end{tikzpicture}
    \caption{The quotient of an icosahedron by a $2\pi/5$ rotation}
    \label{fig:icosahedron}
\end{figure}
\[
\rho_j:\ZZ/5\ZZ\to \CC^*,\quad \rho_j(k)=e^{2\pi ijk/5},\quad j=1,\ldots,4.
\]
The weights of the four corresponding matroids are calculated in Example~\ref{ex:icosahedron3} and are equal to
\[
w(\calM^*(\tX/X,\rho_j))=\begin{cases}
    30-6\sqrt{5}, & j=1,4,\\
    30+6\sqrt{5}, & j=2,3.
\end{cases}
\]
Hence by Theorem~\ref{thm:mainintro} the number of spanning trees of the icosahedral graph is
\[
|\Jac(\tX)|=\frac{1}{5}\cdot 5^1\cdot 1^5\cdot 1^5\cdot 5^1\cdot 2\cdot (30-6\sqrt{5})^2\cdot(30+6\sqrt{5})^2=5184000.
\]

\end{example}

To prove Theorem~\ref{thm:mainintro}, we first upgrade the critical number $|\Jac(X)|$ of a graph to the \emph{Jacobian polynomial} $J_X(x_e)$, whose variables are indexed by the edges of $X$ and whose terms are monomials corresponding to the bases of $\calM(X)$. We similarly replace $|\Jac(\tX)|$ and the weights $w(\calM^*(\tX/X,\rho))$ by polynomials $J_{\tX}(x_e)$ and $P_{\tX/X,\rho}(x_e)$, the latter having coefficients in the real cyclotomic field. We then show in Theorem~\ref{thm:main1} that the polynomial $J_{\tX}$ factors in terms of $J_X$ and the $P_{\tX/X,\rho}$ in the same way as number $|\Jac(\tX)|$ does. We first prove Theorem~\ref{thm:main1} for free covers using the theory of metric $L$-functions, and then derive the non-free case by using an edge contraction procedure (which is the principal reason why we introduce the polynomials). Theorem~\ref{thm:mainintro} is then obtained by simply plugging in $x_e=1$ for all edges $e\in E(X)$.

The paper is organized as follows. In Section~\ref{sec:review} we review the necessary definitions concerning graphs and their Galois covers. In particular, we recall the cohomological description of non-free abelian covers developed in~\cite{2024LenUlirschZakharov}. In Section~\ref{sec:zeta} we recall the relationship between the Ihara zeta function and the graph Jacobian and prove Proposition~\ref{prop:s=1zeta}, which expresses the Jacobian polynomial of a graph in terms of a metric generalization of the zeta function. We recall Zaslavsky's construction of the bias matroid in Section~\ref{sec:matroid}, and in Section~\ref{sec:main} we state and prove the main results and discuss several examples.

\section{Graphs, harmonic morphisms, and Galois covers} \label{sec:review}

\subsection{Graphs and harmonic morphisms} We begin by reviewing a number of basic definitions relating to graphs and their morphisms. We use Serre's definition of a graph (see~\cite{2002Serre}), which is particularly convenient for considering group actions and quotients. 

A \emph{graph} $X$ consists of the following data:
\begin{enumerate}
    \item A finite set of \emph{vertices} $V(X)$.
    \item A finite set of \emph{half-edges} $H(X)$.
    \item A fixed-point-free involution $H(X)\to H(X)$ denoted by $h\mapsto \oh$.
    \item A root map $r:H(X)\to V(X)$.
\end{enumerate}
An \emph{edge} $e=\{h,\oh\}$ of $X$ is an orbit of the involution, and the set of edges is denoted by $E(X)$. The \emph{root vertices} of an edge $e=\{h,\oh\}$ are $r(h)$ and $r(\oh)$, and an edge is called a \emph{loop} if its root vertices coincide. We allow \emph{multiedges}, in other words distinct edges may have the same root vertices. An \emph{oriented edge} $e=(h,\oh)$ is an edge together with a choice of ordering of its half-edges, which determines the \emph{source} and \emph{target} vertices $s(e)=r(h)$ and $t(e)=r(\oh)$. We denote the opposite orientation by $\oe=(\oh,h)$. An \emph{orientation} on a graph $X$ is a choice of orientation for every edge, including the loops, and determines global source and target maps $s:E(X)\to V(X)$ and $t:E(X)\to V(X)$. We consider only finite graphs, and all graphs are connected unless stated otherwise. 

The \emph{tangent space} to a vertex $v\in V(X)$ is
\[
T_v(X)=\{h\in H(X):r(h)=v\},
\]
and the \emph{valency} $\val(v)=|T_v(X)|$ is the number of incoming edges, with loops counted twice. The \emph{genus} of a graph is
\[
g(X)=|E(X)|-|V(X)|+1;
\]
this is not to be confused with the minimum genus of an orientable surface into which $X$ embeds. 

A \emph{morphism} of graphs $f:\tX\to X$ is a pair of maps $f:V(\tX)\to V(X)$ and $f:H(\tX)\to H(X)$ that are equivariant with respect to the involution and root maps. A morphism of graphs maps $f:\tX\to X$ induces a map on the edges $f:E(\tX)\to E(X)$, so in our definition edges may not be contracted to vertices. A morphism of graphs $f:\tX\to X$ is called \emph{harmonic} if there exists a \emph{local degree} function $d_f:V(\tX)\to \ZZ_{>0}$ satisfying the \emph{local balancing condition}: for any vertex $\tv\in V(\tX)$ and any half-edge $h\in T_{f(\tv)}X$ we have
\begin{equation}
\label{eq:localbalancing} 
d_f(\tv)=\left|T_{\tv}(\tX)\cap f^{-1}(h)\right|.
\end{equation}
If $X$ is connected, then a harmonic morphism $f:\tX\to X$ has a \emph{global degree} given by
\[
\deg f=\sum_{\tv\in f^{-1}(v)}d_f(\tv)=|f^{-1}(e)|
\]
for any vertex $v\in V(X)$ or any edge $e\in E(X)$. A harmonic morphism $f$ is called \emph{free} if $d_f(\tv)=1$ for every vertex $\tv\in V(\tX)$. The local balancing condition implies that a free harmonic morphism is a local homeomorphism near every vertex, hence a free harmonic morphism of graphs is the same thing as a covering in the topological sense.

\subsection{Group actions and harmonic $G$-covers} We now recall the theory of harmonic $G$-covers of a graph with abelian Galois group $G$, which was developed in~\cite{2024LenUlirschZakharov}. These covers are a generalization of free Galois covers, described by covering space theory, and are obtained from arbitrary (not necessarily free) group actions on graphs.

Let $X$ be a graph. An \emph{automorphism} of $X$ is an invertible graph morphism $f:X\to X$. An automorphism is harmonic with local degree $d_f(v)=1$ for all $v\in V(X)$ and global degree $\deg f=1$. Conversely, any harmonic morphism of global degree equal to $1$ is an automorphism. The automorphisms of a graph $X$ form a group denoted by $\Aut(X)$. An automorphism of a graph may \emph{flip edges}, in other words may exchange the two half-edges comprising an edge, fixing the edge but exchanging its root vertices. We cannot take the quotient of a graph by such an automorphism\footnote{A quotient by an automorphism that flips edges may be taken in the category of \emph{graphs with legs}, where a leg is a half-edge that is fixed by the involution. We do not develop this perspective here, and refer the interested reader to~\cite{2023MeyerZakharov}.}, so we exclude them from consideration.

\begin{definition} Let $X$ be a graph and let $G$ be a finite group. A \emph{left $G$-action} on $X$ is a homomorphism $G\to \Aut(X)$ such that, for every $g\in G$ and every edge $e=\{h,\oh\}\in E(X)$, one of the following holds:
\begin{enumerate}
    \item $g(h)=h$, hence $g(\oh)=\oh$ and therefore $g(e)=e$.
    \item $g(h)\notin\{h,\oh\}$, and therefore $g(e)\neq e$.
\end{enumerate}
\end{definition}

Given a graph $X$ with a left $G$-action, we define the \emph{quotient graph} $X/G$ to be the set of orbits of the $G$-action. Specifically, the vertices and half-edges of $X/G$ are the orbit sets
\[
V(X/G)=V(X)/G, \quad H(X/G)=H(X)/G,
\]
and the root $H(X/G)\to V(X/G)$ and involution $H(X/G)\to H(X/G)$ maps on $X/G$ are induced from those on $X$. Since the action does not flip edges, the involution map on $X/G$ has no fixed points, therefore $X/G$ is a graph. 

We now consider two distinguished classes of group actions.
\begin{definition} Let $X$ be a graph with a left $G$-action.
\begin{enumerate}
    \item We say that the $G$-action is \emph{free on the edges} if $g(h)\neq h$ for all $h\in H(X)$ and all nontrivial elements $g\in G$ (equivalently, if $g(e)\neq e$ for all $e\in E(X)$ and all nontrivial $g\in G$).
    \item We say that the $G$-action is \emph{free} if $g(v)\neq v$ for all $v\in V(X)$ and all nontrivial $g\in G$.
    
\end{enumerate}   
\end{definition}

Let $X$ be a graph with a left $G$-action. The canonical projection $p:X\to X/G$ sending each vertex and half-edge to its orbit is a graph morphism, but it is not in general harmonic. However, if the action is free on the edges, then $p$ is harmonic if we set the local degrees to be 
\[
d_p(v)=|\Stab(v)|=|\{g\in G:g(v)=v\}|.
\]
for all $v\in V(X)$. The orbit-stabilizer theorem implies the local balancing condition~\eqref{eq:localbalancing} at every vertex of $G$ and that $p$ is a harmonic morphism of global degree $\deg p=|G|$. If in addition the $G$-action is free, then $d_p(v)=1$ for all $v\in V(X)$ and $p$ is a covering space. 

We now consider the morphism $p$ from the perspective of the quotient graph. 

\begin{definition} Let $X$ be a graph and let $G$ be a finite group. A \emph{Galois cover} of $X$ with Galois group $G$, or \emph{harmonic $G$-cover} for short, is a harmonic morphism $p:\tX\to X$ of degree $\deg p=|G|$ together with a $G$-action on $\tX$ such that the following conditions hold:
\begin{enumerate}
    \item The $G$-action is equivariant: for every $g\in G$, $p(g(\te))=\te$ for each edge $\te\in E(\tX)$ and $p(g(\tv))=\tv$ as well as $d_p(g(\tv))=d_p(\tv)$ for each vertex $\tv\in V(\tX)$.
    \item For each $v\in V(X)$, $G$ acts transitively on the fiber $p^{-1}(v)$, and for each $e\in E(X)$, $G$ acts transitively and freely on the fiber $p^{-1}(e)$.
\end{enumerate}
A harmonic $G$-cover is called \emph{free} if it is free as a harmonic morphism, equivalently, if the $G$-action is free on the vertices as well as the edges. The \emph{trivial $G$-cover} is the free disconnected cover $p:X\times G\to X$, where $p$ is projection onto the first factor. An \emph{isomorphism of harmonic $G$-covers} $p:\tX\to X$ and $p':\tX'\to X$ is a $G$-equivariant isomorphism of graphs $i:\tX\to \tX'$ such that $p=p'\circ i$.
\label{def:Gcover}
\end{definition}

It is clear that the two objects defined above are the same. Indeed, if $X$ is a graph with a left $G$-action that is free on the edges, then the quotient map $p:X\to X/G$ is a harmonic $G$-cover. Conversely, if $p:\tX\to X$ is a harmonic $G$-cover, then there is an isomorphism $X\simeq\tX/G$ whose composition with $f$ is the quotient morphism $\tX\to \tX/G$. Furthermore, free $G$-actions correspond to free harmonic $G$-covers (having local degrees $d_p(\tv)=1$ for all $\tv\in V(\tX)$), which explains the choice of terminology.

\begin{example} Consider the graph $\tX$ with two vertices connected by two edges. The automorphism group $\Aut(\tX)$ is the Klein $4$-group generated by $g_1$, which fixes the vertices and exchanges the edges, and $g_2$, which fixes the edges and exchanges the vertices (hence flips both edges). Let $G$ be the subgroup generated by $g_1$, then the quotient $X=\tX/G$ is a graph with two vertices and one edge, and the quotient map $p:\tX\to X$ is a harmonic morphism of degree two (see Figure~\ref{fig:degreetwo}). 

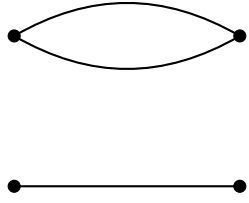
\begin{figure}
    \centering
\begin{tikzpicture}

    \draw[fill](3,0) circle(.08);
    \draw[thick] (3,0) -- (6,0);
    \draw[fill](6,0) circle(.08);
    
\begin{scope}[shift={(0,2)}]

    \draw[fill](3,0) circle(.08);
    \draw[bend left, thick] (3,0) to (6,0);
    \draw[bend right, thick] (3,0) to (6,0);
    \draw[fill](6,0) circle(.08);

\end{scope}
\end{tikzpicture}
    \caption{Harmonic quotient of degree two}
    \label{fig:degreetwo}
\end{figure}
    
\end{example}

    

\subsection{Abelian $G$-covers and dilated cohomology} We now restrict our attention to the case when the group $G$ is abelian\footnote{The description of $G$-covers given in this subsection holds for non-abelian groups as well, provided that one works with non-abelian simplicial cohomology sets.}, and switch to additive group notation. The set of harmonic $G$-covers of a given graph $X$ was given a cohomological interpretation in~\cite{2024LenUlirschZakharov}. This description generalizes the classical description of free $G$-covers in terms of covering space theory, which we recall first. 

Let $p:\tX\to X$ be a free $G$-cover. For each vertex $v\in V(X)$ we choose an equivariant identification of the fiber $p^{-1}(v)$ with $G$ as a set, so that the action is the group operation:
\begin{equation}
\label{eq:freevertices}
p^{-1}(v)=\{\tv_g:g\in G\},\quad g(\tv_{g'})=\tv_{g+g'},
\end{equation}
Now choose an orientation on $X$, and let $s,t:E(X)\to V(X)$ be the source and target maps. For each edge $e\in E(X)$, we equivariantly identify the fiber $p^{-1}(e)$ with $G$ by using the source map $s:p^{-1}(e)\to p^{-1}(s(e))$. The target map $t:p^{-1}(e)\to p^{-1}(t(e))$ is then an equivariant map of $G$-torsors, hence is given by adding a certain element $\eta_e\in G$. In other words,
\begin{equation}
\quad p^{-1}(e)=\{\te_g:g\in G\},\quad g(\te_{g'})=\te_{g+g'},\quad s(\te_g)=\widetilde{s(e)}_g,\quad t(\te_g)=\widetilde{t(e)}_{g+\eta_e}.
\label{eq:freeedges}
\end{equation}
The collection of elements $\{\eta_e\}_{e\in E(X)}$ is called the \emph{$G$-voltage assignment} associated to the free $G$-cover $p$, and depends on the choices of identifications of the vertex fibers $p^{-1}(v)$ with $G$. Choosing different identifications corresponds to a \emph{vertex switching}, which can be conveniently described in terms of simplicial cohomology. Let
\[
C^0(X,G)=G^{V(X)},\quad C^1(X,G)=G^{E(X)}
\]
be the free abelian groups on the vertices and edges of $X$, respectively. We view elements of $G^{V(X)}$ and $G^{E(X)}$ as functions $\xi:V(X)\to G$ and $\eta:E(X)\to G$, respectively. The source and target maps have duals $s^*,t^*:G^{V(X)}\to G^{E(X)}$ given by
\[
s^*(\xi)(e)=\xi(s(e)), \quad t^*(\xi)(e)=\xi(t(e)),
\]
and the first simplicial cohomology group $H^1(X,G)$ is the cokernel of the \emph{coboundary map}:
\begin{equation}
\delta^*:C^0(X,G)\to C^1(X,G),\quad \delta^*=t^*-s^*,\quad H^1(X,G)=G^{E(X)}/\Im \delta^*.
\label{eq:coboundaryfree}
\end{equation}
The $G$-voltage assignment $\eta\in C^1(X,G)$ of a free $G$-cover $p:\tX\to X$ defines a cocycle $[\eta]\in H^1(X,G)$ that does not depend on any choices, and two free $G$-covers $p:\tX\to X$ and $p':\tX'\to X$ with $G$-voltage assignments $\eta$ and $\eta'$ are isomorphic if and only if $[\eta]=[\eta']$ in $H^1(X,G)$. Conversely, given a simplicial cocycle $[\eta]\in H^1(X,G)$ and a choice of representative $\eta\in  C^1(X,G)$, we can construct a free $G$-cover $p:\tX\to X$ having $G$-voltage assignment $\eta$ using~\eqref{eq:freevertices} and~\eqref{eq:freeedges}.

This cohomological description of $G$-covers was generalized from the free to the harmonic case in the paper~\cite{2024LenUlirschZakharov} (the interested reader is referred to the original arXiv version, which was substantially shortened for publication). In a nutshell, we modify the definition of simplicial cohomology by taking appropriate quotients by the vertex stabilizer subgroups. Given a harmonic $G$-cover $p:\tX\to X$, we first record the data of the stabilizers.

\begin{definition} Let $p:\tX\to X$ be a harmonic $G$-cover, where $G$ is abelian. The \emph{dilation subgroup} of a vertex $v\in V(X)$ is
\[
D(v)=\{g\in G:g(\tv)=\tv\mbox{ for all }\tv\in p^{-1}(v)\}.
\]
We say that a vertex $v\in V(X)$ is \emph{free} if $D(v)$ is the trivial subgroup and \emph{dilated} otherwise. The \emph{$G$-dilation datum} of $p$ is the collection of subgroups $D(v)\subseteq G$ for all $v\in V(X)$. 

\end{definition}
Note that, since $G$ is abelian, the dilation subgroup $D(v)$ of a vertex $v\in V(X)$ is the stabilizer of any (and hence every) preimage vertex $\tv\in p^{-1}(v)$. We now consider a harmonic $G$-cover $p:\tX\to X$ with $G$-dilation datum $\{D(v)\}_{v\in V(X)}$, and choose an orientation on $X$. Consider the following groups:
\[
C^0(X,D)=\bigoplus_{v\in V(X)}G/D(v),\quad C^1(X,D)=
\bigoplus_{e\in E(X)}G/(D(t(e))+D(s(e))).
\]
An element $\eta\in C^1(X,D)$ may be thought of as a \emph{dilated $G$-voltage assignment} on $X$, in other words a choice of an element $\eta(e)\in G$ along every edge $e\in E(X)$, modulo the dilation subgroups $D(t(e))$ and $D(s(e))$ at the root vertices. As before, the source and target maps have duals 
\[
s^*_D,t^*_D:C^0(X,D)\to C^1(X,D),\quad s^*_D(\xi)(e)=s_e(\xi(s(e))), \quad t^*_D(\xi)(e)=t_e(\xi(t(e))),
\]
where for $e\in E(X)$ we denote by
\[
s_e:G/D(s(e))\to G/(D(s(e))+D(t(e))),\quad t_e:G/D(t(e))\to G/(D(s(e))+D(t(e)))
\]
the natural quotient maps.

\begin{definition} The \emph{dilated cohomology group} $H^1(X,D)$ is the cokernel of the coboundary map
\[
\delta_D^*:C^0(X,D)\to C^1(X,D),\quad \delta^*_D=t^*_D-s^*_D,\quad H^1(X,D)=C^1(X,D)/\Im\delta^*_D.
\]
\end{definition}

The class $[\eta]\in H^1(X,D)$ of a dilated $G$-voltage assignment $\eta\in C^1(X,D)$ is its \emph{vertex switching equivalence} class. Just as in the free case, this describes the set of isomorphism classes of $G$-covers\footnote{The paper~\cite{2024LenUlirschZakharov} deals with the more general situation where the $G$-action may have nontrivial edge stabilizers.}.

\begin{theorem}[Theorem 2.3 in~\cite{2024LenUlirschZakharov}] Let $X$ be a graph, let $G$ be a finite abelian group, and let $D=\{D(v)\subseteq G\}_{v\in V(X)}$ be a $G$-dilation datum on $X$. There is a natural one-to-one correspondence between harmonic $G$-covers $p:\tX\to X$ having $G$-dilation datum $D$ and elements of $H^1(X,D)$.
\label{thm:LUZ}
\end{theorem}

This theorem is proved by identifying harmonic $G$-covers with torsors over an appropriate constructible sheaf of abelian groups on $X$. We instead give an explicit description of the harmonic $G$-cover $p:\tX\to X$ corresponding to the equivalence class $[\eta]\in H^1(X,D)$ of a dilated $G$-voltage assignment $\eta\in C^1(X,D)$.

For each vertex $v\in V(X)$, we define the fiber $p^{-1}(v)$ to be the quotient group $G/D(v)$, so that the $G$-action is induced by the group operation:
\begin{equation}
p^{-1}(v)=\{\tv_{[g]}:[g]\in G/D(v)\},\quad g(\tv_{[g']})=\tv_{[g+g']}\mbox{ for }g\in G\mbox{ and }[g']\in G/D(v).
\label{eq:dilatedvertexfiber}
\end{equation}
Here $[g]$ denotes the equivalence class of $g\in G$ in $G/D(v)$. Similarly, for an edge $e\in E(X)$ we identify the fiber $p^{-1}(e)$ with the group $G$ itself (since the edge stabilizers are trivial) and the $G$-action with the group operation:
\begin{equation}
p^{-1}(e)=\{\te_{g}:g\in G\},\quad g(\te_{g'})=\te_{g+g'}\mbox{ for }g,g'\in G.
\label{eq:dilatededgefiber}
\end{equation}
The vertex and edge sets of $\tX$ are
\[
V(\tX)=\bigsqcup_{v\in V(X)}p^{-1}(v),\quad
E(\tX)=\bigsqcup_{e\in E(X)}p^{-1}(e),
\]
and the map $p:\tX\to X$ sends $p^{-1}(v)$ and $p^{-1}(e)$ to $v$ and $e$, respectively. To complete the definition of the $G$-cover $p:\tX\to X$, we need to connect the edges to the vertices, in other words we need to define the source and target maps $s,t:E(\tX)\to V(\tX)$. Just as in~\eqref{eq:freeedges}, we define the source map trivially:
\begin{equation}
s:p^{-1}(e)\to p^{-1}(s(e)),\quad s(\te_{g})=\widetilde{s(e)}_{[g]}.
\label{eq:dilatedsourcefiber}
\end{equation}
The target map, however, is twisted by the dilated $G$-voltage assignment $[\eta]\in H^1(X,D)$. Specifically, choose a representative $\eta\in C^1(X,D)$ of $[\eta]$, and, for each $e\in E(X)$, choose a preimage of $\eta_e\in G/[D(s(e))+D(t(e))]$ in $G$, which we also denote by $\eta_e$ by abuse of notation. We then set
\begin{equation}
t:p^{-1}(e)\to p^{-1}(t(e)),\quad t(\te_{g})=\widetilde{t(e)}_{[g+\eta_e]}.
\label{eq:dilatedtargetfiber}
\end{equation}
Theorem~\ref{thm:LUZ} states that the isomorphism class of the resulting harmonic $G$-cover $p:\tX\to X$ does not depend on any choices made, that any cover may be obtained in this way, and that two covers are isomorphic if and only if the corresponding elements in $H^1(X,D)$ are equal.

\begin{example} \label{ex:icosahedron2} Consider $G=\ZZ/5\ZZ$ acting on the icosahedron $\tX$ by a $2\pi/5$ rotation about an axis through two opposite vertices. The quotient $p:\tX\to X$ is shown on Figure~\ref{fig:icosahedron}, and the dilation groups are
\[
D(v_1)=D(v_4)=\ZZ/5\ZZ,\quad D(v_2)=D(v_3)=0.
\]
The group $C^0(X,D)\simeq (\ZZ/5\ZZ)^2$ has two generators, corresponding to $v_2$ and $v_3$, while $C^1(X,D)\simeq (\ZZ/5\ZZ)^4$ has four generators corresponding to $e_2$, $e_3$, $e_4$, and $e_5$. We see that the dilated cohomology group is $H^1(X,D)=(\ZZ/5\ZZ)^3$, and in fact coincides with the usual simplicial cohomology group with coefficients in $\ZZ/5\ZZ$ (this happens because the dilated vertices are extremal). The dilated voltage assignment $\eta$ describing the cover $p:\tX\to X$ takes values in $\ZZ/5\ZZ$ on $e_2,\ldots,e_5$ and is trivial on $e_1$ and $e_6$, and we may choose it to be 
\begin{equation}
\eta(e_2)=1,\quad \eta(e_3)=1,\quad \eta(e_4)=0,\quad \eta(e_5)=1.
\label{eq:icosahedronvoltage}
\end{equation}

\end{example}

\subsection{Contraction and resolution} A simple but important technical tool for us is the operation of \emph{edge contraction}. Given a set of edges $F\subseteq E(X)$ of a graph $X$, the \emph{induced subgraph} $X[F]\subseteq X$ is the minimal subgraph of $X$ containing $F$, in other words $E(X[F])=F$ and $V(X[F])=r(F)$. Now let $X[F]=X_1\sqcup \cdots\sqcup X_k$ be the decomposition of $X[F]$ into connected components. The \emph{contraction} $X_F$ of $X$ along $F$ is the graph obtained by contracting each $X_i$ to a separate vertex $v_i$. Specifically,
\[
V(X_F)=[V(X)\backslash V(X[F])]\cup\{v_1,\ldots,v_k\},\quad E(X_F)=E(X)\backslash F,
\]
and any edge $e\in E(X)\backslash F$ that in $X$ is rooted at a vertex of one of the $X_i$ is instead rooted at the corresponding vertex $v_i$ of $X_F$. 

We can likewise define the contraction of a harmonic morphism along a set of edges of the target graph. Let $f:\tX\to X$ be a harmonic morphism, and let $F\subseteq E(X)$. The \emph{contraction} $f_F:\tX_{f^{-1}(F)}\to X_F$ of $f$ along $F$ is defined as follows. First, we contract the edges $F$ of $X$ and their preimages $f^{-1}(F)$ in $\tX$ to obtain the graphs $X_F$ and $\tX_{f^{-1}(F)}$, respectively. Outside of the contracted vertices of $\tX_{f^{-1}(F)}$, the map $f_F$ is the restriction of $f$. Now let $X_i\subseteq X$ be a connected component of $X[F]$, and let $f^{-1}(X_i)=\tX_{i1}\sqcup \cdots\sqcup \tX_{ik_i}$ be the connected components of its preimage. Each $\tX_{ij}$ corresponds to a contracted vertex $\tv_{ij}\in V(\tX_{f^{-1}(F)})$, and we set $f_F(\tv_{ij})=v_i$, where $v_i\in V(X_F)$ is the vertex corresponding to $X_i$. It is elementary to check that $f_F$ is a harmonic morphism, provided that we set $d_{f_F}(\tv_{ij})=\deg f|_{\tX_{ij}}$, in other words the local degree of $f_F$ at $\tv_{ij}$ is equal to the global degree of $f$ on the connected graph $\tX_{ij}$.

Finally, let $p:\tX\to X$ be a harmonic $G$-cover and let $F\subseteq E(X)$ be a set of edges of the target. We can naturally equip the contraction $p_F:\tX_{p^{-1}(F)}\to X_F$ with the structure of a harmonic $G$-cover. Indeed, for a vertex $v\in V(X)\backslash V(X[F])$ the fibers $p^{-1}(v)$ and $p_F^{-1}(v)$ are the same, which defines the $G$-action on $p_F^{-1}(v)$. The same is true for any edge $e\in E(X)\backslash F$ of the contracted graph. Finally, let $v_i\in V(X_F)$ be a vertex corresponding to a connected component $X_i$ of $X[F]$. Automorphisms preserve connectivity, therefore $G$ acts on the connected components $\tX_{ij}$ of $p^{-1}(X_i)$, and the action is transitive because it is so on the fiber over any vertex of $X_i$. This defines the $G$-action on the fiber $p_F^{-1}(v_i)$.

We now show that any harmonic $G$-cover $p:\tX\to X$ can be obtained by edge contraction from a free $G$-cover, which we call a \emph{free resolution} of $p$. For simplicity, we restrict our attention to abelian covers.

\begin{lemma} Let $p:\tX\to X$ be a harmonic $G$-cover with abelian Galois group $G$. There exists a free $G$-cover $p_f:\tX_f\to X_f$ such that $p$ is an edge contraction of $p_f$.
\label{lem:contraction}
\end{lemma}

\begin{proof} Choose a dilated vertex $v\in V(X)$ and let $e_1,\ldots,e_a$ and $f_1,\ldots,f_b$ be respectively the non-loop and loop edges at $v$. We recall the local description~\eqref{eq:dilatedvertexfiber}-\eqref{eq:dilatedtargetfiber} of $p$ near $v$. The fiber $p^{-1}(v)$ is identified with the quotient $G/D(v)$ as a $G$-set, while all edge fibers are identified with $G$:
\[
p^{-1}(v)=\{\tv_{[g]}:[g]\in G/D(v)\},\quad p^{-1}(e_i)=\{\te_{i,g}:g\in G\},\quad p^{-1}(f_j)=\{\tf_{j,g}:g\in G\}.
\]
We may assume that all non-loops are oriented away from $v$, so that $s(e_i)=s(f_j)=v$. The corresponding source maps on $\tX$ are the quotient map $G\to G/D(v)$:
\begin{equation}
s:p^{-1}(e_i)\to p^{-1}(v),\quad s(\te_{i,g})=\tv_{[g]},\quad  s:p^{-1}(f_j)\to p^{-1}(v),\quad s(\tf_{j,g})=\tv_{[g]}.
\label{eq:localfiber1}
\end{equation}
Finally, for each loop $f_j$ at $v$ there is an element $[\eta_j]\in G/(D(s(f_j))+D(t(f_j)))=G/D(v)$ such that the target map on the corresponding fiber consists of taking the quotient and adding $[\eta_j]$:
\begin{equation}
t:p^{-1}(f_j)\to p^{-1}(v),\quad t(\tf_{j,g})=\tv_{[g+\eta_j]}.
\label{eq:localfiber2}
\end{equation}

We now choose generators $g_1,\ldots,g_c\in D(v)$ of the dilation subgroup and a representative $\eta_j\in G$ of each $[\eta_j]$. We construct a harmonic $G$-cover $p':\tX'\to X'$ as follows. The graph $X'$ consists of the graph $X$ with $c$ loops $l_1,\ldots,l_c$ attached to $v$. To construct $\tX'$, we first replace the fiber $p^{-1}(v)=G/D(v)$ with a copy of $G$, so that
\[
(p')^{-1}(v)=\{\tv_g:g\in G\}.
\]
We then reattach the loose edges in $p^{-1}(e_i)$ and $p^{-1}(f_j)$ using the same formulas~\eqref{eq:localfiber1} and~\eqref{eq:localfiber2}, but without taking the quotient by $D(v)$. Finally, the fiber over each loop $l_k$ is a copy of $G$, the source map is trivial, and the target map is determined by $g_k$:
\[
(p')^{-1}(l_k)=\{\widetilde{l}_{k,g}:g\in G\},\quad s(\widetilde{l}_{k,g})=\tv_g,\quad t(\widetilde{l}_{k,g})=\tv_{g+g_k}.
\]
The resulting morphism $p':\tX'\to X'$ is a harmonic $G$-cover with trivial dilation subgroup at $v$. Furthermore, let $X_v\subset X'$ be the subgraph consisting of $v$ and the new loops $l_k$. Edge paths in the graph $(p')^{-1}(X_v)$ correspond to linear combinations of the chosen generators $g_k$ of $G$, hence two vertices $\tv_g,\tv_{g'}\in (p')^{-1}(v)$ lie in the same connected component of $(p')^{-1}(X_v)$ if and only if $g-g'\in D(v)$. Hence the edge contraction of $p'$ along the added loops $l_k$ is the original $G$-cover $p$. Performing this resolution at each dilated vertex of $X$, we obtain a free resolution $p_f$.

An example of the local picture of the free resolution is given on Figure~\ref{fig:freeresolution}. The Galois group is $G=\ZZ/6\ZZ$, and the dilated vertex $v$ has dilation subgroup $D(v)=\ZZ/3\ZZ$.

\end{proof}

\begin{figure}
    \centering
\begin{tikzpicture}

    \draw[thick] (-1.5,0) -- (1.5,0);
    \node[below] (0,0) {$v$};
    \draw[fill](0,0) circle(.08);

\begin{scope}[shift={(0,2)}]

    \draw[fill](0,0) circle(.08);
    \draw[thick] (-1.5,1) -- (1.5,-1);
    \draw[thick] (-1.5,-1) -- (1.5,1);
    \draw[thick] (-1.5,0) -- (1.5,0);

\end{scope}
\begin{scope}[shift={(0,5)}]

    \draw[fill](0,0) circle(.08);
    \draw[thick] (-1.5,1) -- (1.5,-1);
    \draw[thick] (-1.5,-1) -- (1.5,1);
    \draw[thick] (-1.5,0) -- (1.5,0);

\end{scope}

\begin{scope}[shift={(4,0)}]

    \draw[thick] (-1.5,0) -- (1.5,0);
    \draw[thick] (0,-0.5) circle(0.5);
    \draw[fill](0,0) circle(.08);
\end{scope}

\begin{scope}[shift={(4,2)}]
    \draw[fill](-1,0) circle(.08);
    \draw[fill](0.5,1) circle(.08);
    \draw[fill](0.5,-1) circle(.08);
    \draw[thick] (-1.5,1) -- (1.5,1);
    \draw[thick] (-1.5,-1) -- (1.5,-1);
    \draw[thick] (-1.5,0) -- (1.5,0);
    \draw[thick] (-1,0) -- (0.5,1);
    \draw[thick] (0.5,1) -- (0.5,-1);
    \draw[thick] (0.5,-1) -- (-1,0);

\end{scope}

\begin{scope}[shift={(4,5)}]
    \draw[fill](-1,0) circle(.08);
    \draw[fill](0.5,1) circle(.08);
    \draw[fill](0.5,-1) circle(.08);
    \draw[thick] (-1.5,1) -- (1.5,1);
    \draw[thick] (-1.5,-1) -- (1.5,-1);
    \draw[thick] (-1.5,0) -- (1.5,0);
    \draw[thick] (-1,0) -- (0.5,1);
    \draw[thick] (0.5,1) -- (0.5,-1);
    \draw[thick] (0.5,-1) -- (-1,0);

\end{scope}

\end{tikzpicture}
    \caption{Resolving a vertex with dilation group $D(v)=\ZZ/3\ZZ\subset \ZZ/6\ZZ$}
    \label{fig:freeresolution}
\end{figure}
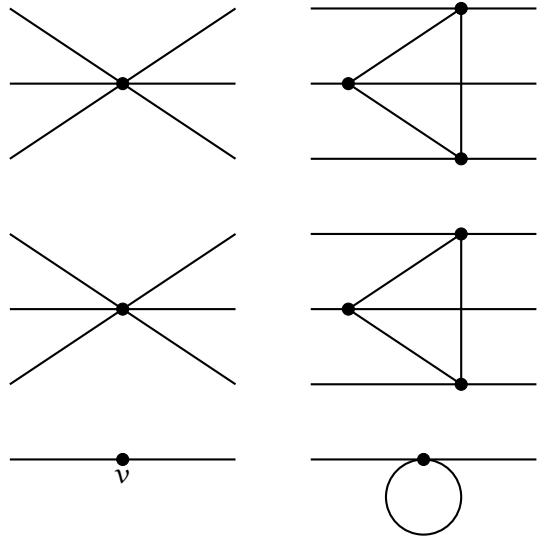



\section{Jacobian polynomials and zeta functions} \label{sec:zeta}

In this section, we recall the theory of chip-firing on a graph $X$, the critical group of $X$ (which we call the \emph{Jacobian group} and denote by $\Jac(X)$), and the relationship to the spanning trees of $X$. Our main tool for computing the order of $\Jac(X)$ (which is the number of spanning trees of $X$) is the Ihara zeta function $\zeta(s,X)$, and we recall the theory of zeta functions and the more general Artin--Ihara $L$-functions of free $G$-covers.

\subsection{Chip-firing, the graph Laplacian, and the Jacobian group} Let $\Div(X)=\ZZ^{V(X)}$ be the free abelian group on the vertices of a graph $X$. An element of $\Div(X)$ is called a \emph{divisor} and may interpreted as an assignment of a number of chips (positive, negative, or zero) to each vertex of $X$. The \emph{Laplacian operator} $L$ of $X$ is the $\ZZ$-linear operator on $\ZZ^{V(X)}$ defined as follows. Let $n$ be the number of vertices of $X$, and define the $n\times n$ \emph{valency} and \emph{adjacency} matrices $Q$ and $A$ as follows:
\begin{equation}
Q_{uv}=\begin{cases} \val(u), & u=v, \\
0, & u\neq v,
\end{cases}\quad
A_{uv}=\mbox{number of edges between }u\mbox{ and }v.
\label{eq:QA}
\end{equation}
We note that each loop at a vertex $u$ contributes $2$ to both $Q_{uu}$ and $A_{uu}$. The matrix of the Laplacian operator is the difference of the valency and adjacency matrices:
\[
L:\ZZ^{V(X)}\to \Div(X),\quad L_{uv}=Q_{uv}-A_{uv}.
\]
Specifically, for $v\in V(X)$ the divisor $-L(v)$ is obtained by \emph{firing} the vertex $v$, in other words by moving a chip from $v$ along each edge $e\in T_vX$ to its other root vertex (the chip returns to $v$ if $e$ is a loop). The image of $L$ is the subgroup of \emph{principal divisors}, which lies in $\Div_0(X)$, the subgroup of divisors with total number of chips is equal to zero. The corresponding quotient is a finite group called the \emph{Jacobian}, or \emph{critical group}, of $X$:
\[
\Jac(X)=\Div_0(X)/\Im L.
\]

Given a morphism $f:\tX\to X$ of graphs, we define a \emph{pushforward map}
\[
f_*:\Div(\tX)\to \Div(X),\quad f_*\left[\sum_{\tv\in V(\tX)}a_{\tv}\tv\right]=\sum_{\tv\in V(\tX)}a_{\tv}f(\tv).
\]
This map does not in general respect chip-firing, and there is no relationship between $\Jac(\tX)$ and $\Jac(X)$. However, if $f$ is a harmonic morphism, then
\[
(f_*\circ L_{\tX})(\tv)=d_f(\tv)(L_X\circ f_*)(\tv)
\]
for any $\tv\in V(\tX)$, where $L_{\tX}$ and $L_X$ are the Laplacians of $\tX$ and $X$, respectively. Hence the image of a principal divisor is principal, and $f_*$ induces a homomorphism on the Jacobian groups:
\[
f_*:\Jac(\tX)\to \Jac(X).
\]
The map $f_*$ is surjective and hence
\[
|\Ker f_*|=\frac{|\Jac(\tX)|}{|\Jac(X)|}.
\]
The purpose of this paper is to evaluate this quantity in the case when $f$ is an abelian harmonic $G$-cover.

\subsection{Kirchhoff's theorem and the Jacobian polynomial} The order of the Jacobian of a graph $X$ can be computed from the Laplacian matrix $L$, which is always singular: its rows sum to zero and so do its columns. Kirchhoff's matrix tree theorem states that any cofactor of $L$ is equal to the number of spanning trees of $X$. It is elementary to see that this quantity is also the order of the Jacobian group of $X$:
\begin{equation}
|\Jac(X)|=\mbox{number of spanning trees of }X=\sum_{T\subseteq X}1,
\label{eq:matrixtree}
\end{equation}
where the sum is taken over all spanning trees $T$ of $X$.

We gain a certain flexibility by converting this quantity into a polynomial. Let $x_e$ be a variable for each edge $e\in E(X)$, representing the length of $e$. We now weigh the terms in the sum~\eqref{eq:matrixtree} by the products of the lengths of the complementary edges:

\begin{definition} The \emph{Jacobian polynomial} of a graph $X$ is
\begin{equation}
J_X(x_e)=\sum_{T\subseteq X}\prod_{e\in E(X)\backslash E(T)}x_e\in \ZZ[x_e]_{e\in E(X)},
\label{eq:JX}
\end{equation}
where the sum is taken over all spanning trees $T$ of $X$.
\end{definition}
The number of edges in the complement of any spanning tree of $X$ is equal to the genus $g(X)$, hence $J(X)$ is a homogeneous polynomial of degree $g(X)$ with integer coefficients. Kirchhoff's matrix tree theorem~\eqref{eq:matrixtree} then states that
\[
|\Jac(X)|=J_X(1).
\]

More generally, let $\un=\{n_e\}_{e\in E(X)}$ be positive integers indexed by the edges of $X$, and let $X_{\un}$ be the graph obtained from $X$ by replacing each edge $e\in E(X)$ with a chain of $n_e$ edges. A moment's thought shows that plugging $x_e=n_e$ into the Jacobian polynomial of $X$ counts the spanning trees of $X_{\un}$, so
\begin{equation}
|\Jac(X_{\un})|=\mbox{number of spanning trees of }X_{\un}=J_X(n_e).
\label{eq:Jn}
\end{equation}
In other words, $J_X$ computes the order of the Jacobian of any graph whose topological type is the same as $X$.

Even more generally, let $\ell:E(X)\to \RR_{>0}$ be arbitrary real numbers, and let $\Xi=(X,\ell)$ be the \emph{metric graph} obtained by identifying each edge $e\in E(X)$ with a real interval of length $\ell(e)$. The theory of chip-firing on metric graphs is developed analogously, where the divisors may now be supported on the interior points of edges. We do not review this theory here, and the interested reader is referred to~\cite{2008MikhalkinZharkov}. We only note that the Jacobian $\Jac(\Xi)$ of a metric graph $\Xi=(X,\ell)$ is a real torus of dimension $g(X)$, and has a canonically defined volume that is computed by the Jacobian polynomial:
\begin{proposition}[Theorem 5.2 in~\cite{2014AnBakerKuperbergShokrieh}] The volume of the Jacobian $\Jac(\Xi)$ of a metric graph $\Xi=(X,\ell)$ is given by
\[
\Vol^2(\Jac(\Xi))=\sum_{T\subseteq X}\prod_{e\in E(X)\backslash E(T)}\ell(e)=J_X(\ell(e)).
\]
    
\end{proposition}

Finally, we have the following useful lemma, which shows that Jacobian polynomials behave well under edge contraction.

\begin{lemma} Let $X$ be a graph, let $e\in E(X)$ be an edge, and let $X_e$ be the graph obtained by contracting $e$. The Jacobian polynomial of $X_e$ is equal to
\begin{equation}
J_{X_e}=\begin{cases} J_X|_{x_e=0}, & e\mbox{ is not a loop,}\\
J_X/x_e, & e\mbox{ is a loop.}
\end{cases}
\label{eq:Jedgecontraction}
\end{equation}
\end{lemma}

\begin{proof} If $e$ is not a loop, then there is a natural bijection between the spanning trees of $X$ containing $e$ and the spanning trees of $X_e$, and setting $x_e=0$ removes the terms of $J_X$ corresponding to the spanning trees of $X$ that do not contain $e$. If $e$ is a loop, then it is contained in the complement of any spanning tree of $X$, so $J_X$ divides $x_e$. The graph $X_e$ has the same set of spanning trees, but without $e$ as a complementary edge, hence $J_{X_e}=J_X/x_e$.
    
\end{proof}

\subsection{Zeta functions and $L$-functions} Our calculations of Jacobian polynomials in $G$-covers are based on the theory of zeta functions (for an individual graph $X$) and $L$-functions (for a free $G$-cover $p:\tX\to X$). These functions are defined as products over the set of primes of $X$, which are equivalence classes of certain closed paths on $X$. We recall this theory, following~\cite{2010Terras}, and to simplify exposition we consider only abelian $G$. 

Let $X$ be a graph. A \emph{path} on $X$ is a sequence $P=e_1\cdots e_k$ of oriented edges such that $t(e_i)=s(e_{i+1})$ for all $i=1,\ldots,k-1$. We say that $P$ is \emph{closed} if $t(e_k)=s(e_1)$, and a closed path $P$ is \emph{reduced} if $e_{i+1}\neq \oe_i$ for all $i=1,\ldots,k-1$ and $e_k\neq \oe_1$. A closed reduced path $P$ is called \emph{primitive} if $P\neq Q^c$ for any closed path $Q$ and any $c\geq 2$. Given a closed reduced path $P=e_1\cdots e_n$, any path of the form $e_k\cdots e_ne_1\cdots e_{k-1}$ for $k=1,\ldots,n$ (consisting of the same sequence of oriented edges, but with a possibly different starting vertex) is \emph{equivalent} to $P$, and a \emph{prime} of $X$ is an equivalence class $[P]$ of a primitive path $P$. We denote by $\mathfrak{P}(X)$ the set of primes of $X$, which is usually infinite.

We consider three zeta functions associated to a graph $X$, in order of increasing specialization. The \emph{edge zeta} is the most general one, but mostly plays an auxiliary role. The \emph{metric zeta} is the most important for us, and is obtained from the edge zeta by specialization. Further specializing, we obtain the \emph{Ihara zeta}, for which we have the most convenient identities.

\begin{definition} Let $X$ be a graph. 

\begin{enumerate} \item Let $\{w_{ef}\}$ be a set of variables indexed by pairs of oriented edges $e$ and $f$ such that $t(e)=s(f)$ and $f\neq \oe$. The \emph{edge norm} $N(P)$ of a closed path $P=e_1\cdots e_k$ on $X$ is
\[
N(P)=w_{e_1e_2}w_{e_2e_3}\cdots w_{e_{k-1}e_k}w_{e_ke_1}.
\]
The \emph{edge zeta function} of $X$ is
\[
\zeta_E(w_{ef},X)=\prod_{[P]\in \mathfrak{P}(X)}(1-N(P))^{-1}.
\]

\item Let $\{x_e\}$ be a set of variables indexed by the unoriented edges of $X$, where we view $x_e$ as the length of $e$ (so that $x_{\oe}=x_e$), and let $s$ be a complex variable. The \emph{metric length} of a closed path $P=e_1\cdots e_k$ on $X$ is
\[
M(P)=x_{e_1}+\cdots+x_{e_k}.
\]
The \emph{metric zeta function} of $X$ is
\[
\zeta_M(s,x_e,X)=\prod_{[P]\in \mathfrak{P}(X)}\left(1-s^{M(P)}\right)^{-1}.
\]

\item Let $s$ be a complex variable. The \emph{Ihara zeta function} of $X$ is
\[
\zeta(s,X)=\prod_{[P]\in \mathfrak{P}(X)}\left(1-s^{L(P)}\right)^{-1},
\]
where $L(P)$ is the number of edges in $P$.
\end{enumerate}
\label{def:zeta}
\end{definition}

It is clear that the edge zeta function specializes to the metric zeta function, which in turn specializes to the Ihara zeta function:
\begin{equation}
\zeta_M(s,x_e,X)=\left.\zeta_E(w_{ef},X)\right|_{w_{ef}=s^{x_e}},\quad \zeta(s,X)=\left.\zeta_M(s,x_e,X)\right|_{x_e=1}.
\label{eq:specializationzeta}
\end{equation}
The three zeta functions may be computed as reciprocals of explicit matrix determinants. Let $m$ be the number of edges of $X$, and let $W$ be the $2m\times 2m$-matrix indexed by the oriented edges of $X$, with entry $w_{ef}$ if $t(e)=s(f)$ and $f\neq \oe$, and zero otherwise. The \emph{two-term determinant formula} computes the edge zeta function in terms of $W$.

\begin{theorem}[Theorem 3 in~\cite{1996StarkTerras} or Theorem 11.4 in~\cite{2010Terras}] Let $X$ be a graph. The edge zeta function of $X$ satisfies 
\begin{equation}
\zeta_E(w_{ef},X)^{-1}=\det(I-W).
\label{eq:twotermzeta}
\end{equation}
\label{thm:edgezeta}  
\end{theorem}
We can use the above formula and the specializations~\eqref{eq:specializationzeta} to compute the metric $\zeta_M(s,x_e,X)$ and the Ihara $\zeta(s,X)$ zeta functions. However, the latter admits a more convenient \emph{three-term determinant formula} in terms of the valency and adjacency matrices~\eqref{eq:QA}:
\begin{theorem} [See~\cite{1992Bass} or Theorem 2.5 in~\cite{2010Terras}] Let $X$ be a graph. The Ihara zeta function of $X$ satisfies
\begin{equation}
\zeta(s,X)^{-1}=(1-s^2)^{g(X)-1}\det[I_n-sA+s^2(Q-I_n)].
\label{eq:threetermzeta}
\end{equation}
\label{thm:Iharazeta}
\end{theorem}

We now consider a free $G$-cover $p:\tX\to X$ with abelian Galois group $G$. Choosing an orientation on $X$, the cover corresponds to a $G$-voltage assignment $[\eta]\in H^1(X,G)$, which we extend the assignment to edges with the opposite orientation by setting $\eta(\oe)=-\eta(e)$ for all $e\in E(G)$. Given a closed path $P=e_1\cdots e_k$ on $X$, the element
\[
\eta(P)=\eta(e_1)+\cdots+\eta(e_k)
\]
is called the \emph{Frobenius element} of $P$ and has the following geometric significance. Let $v\in V(X)$ be the starting and ending vertex of $P$, and fix a preimage $\tv\in p^{-1}(v)$. Since $f$ is a covering, the path $P$ has a unique lift to a path $Q$ on $\tX$ that starts at $\tv$. The path $Q$ need not be closed, but its end vertex $\tv'$ also lies in $p^{-1}(v)$, and $\eta(P)$ is the unique element of $G$ that maps $\tv$ to $\tv'$. Finally, it is clear that equivalent paths have equal Frobenius elements.

We now fix a character $\rho$ of the Galois group $G$ (since the latter is abelian, this is the same thing as an irreducible representation), and consider the same products over the primes of $X$ that define the zeta functions, but weighted by the Frobenius elements. The variables $w_{ef}$, $x_e$, and $s$ and the lengths $N(P)$, $M(P)$, and $L(P)$ are the same as in Definition~\ref{def:zeta}.

\begin{definition} Let $p:\tX\to X$ be a free $G$-cover with abelian Galois group $G$ and let $\rho:G\to \mathbb{C}^*$ be a character of $G$. 

\begin{enumerate} \item The \emph{edge Artin $L$-function} of $p:\tX\to X$ and $\rho$ is
\[
L_E(w_{ef},\tX/X,\rho)=\prod_{[P]\in \mathfrak{P}(X)}(1-\rho(\eta(P))N(P))^{-1}.
\]

\item The \emph{metric Artin $L$-function} of $p:\tX\to X$ and $\rho$ is
\[
L_M(s,x_e,\tX/X,\rho)=\prod_{[P]\in \mathfrak{P}(X)}\left(1-\rho(\eta(P))s^{M(P)}\right)^{-1}.
\]

\item The \emph{Artin--Ihara $L$ function} of $p:\tX\to X$ and $\rho$ is
\[
L(s,\tX/X,\rho)=\prod_{[P]\in \mathfrak{P}(X)}\left(1-\rho(\eta(P))s^{L(P)}\right)^{-1}.
\]
\end{enumerate}
\label{def:L}
\end{definition}

As for the zeta functions, we have a specialization rule
\begin{equation}
L_M(s,x_e,\tX/X,\rho)=\left.L_E(w_{ef},\tX/X,\rho)\right|_{w_{ef}=s^{x_e}},\quad L(s,\tX/X,\rho)=\left.L_M(s,x_e,\tX/X,\rho)\right|_{x_e=1}.
\label{eq:specializationL}
\end{equation}
The $L$-functions can also be computed by two- and three-term determinant formulas. As before, let $m$ be the number of edges of $X$, and let $W
_{\rho}$ be the $2m\times 2m$-matrix indexed by the oriented edges of $X$, with entry $w_{ef}\rho(\eta(e))$ if $t(e)=s(f)$ and $f\neq \oe$, and zero otherwise. 

\begin{theorem}[Theorem 19.3 in~\cite{2010Terras}] Let $p:\tX\to X$ be a free $G$-cover with abelian Galois group $G$ and let $\rho:G\to \mathbb{C}^*$ be a character of $G$. The edge Artin $L$-function of $X$ satisfies 
\begin{equation}
L_E(w_{ef},\tX/X,\rho)^{-1}=\det(I-W_{\rho}).
\label{eq:twotermL}
\end{equation}
\label{thm:edgeL}  
\end{theorem}

There is an analogous three-term determinant formula for the Artin--Ihara $L$-function. Let $n$ be the number of vertices of $X$, and let $Q$ be the $n\times n$ valency matrix~\eqref{eq:QA}. Choose an orientation on $X$ and define the $n\times n$ \emph{$\rho$-twisted adjacency matrix} $A_{\rho}$ as follows:
\[
(A_{\rho})_{uv}=\sum_{e:s(e)=u,t(e)=v}\rho(\eta(e))+\sum_{e:s(e)=v,t(e)=u}\overline{\rho(\eta(e))}.
\]
We emphasize that the sum is taken using the chosen orientation, in other words each edge contributes only once. We also note that $A_{\rho}$ does not depend on the choice of orientation, because $\rho:G\to \mathbb{C}^*$ takes values in the unit circle and therefore
\[
\rho(\eta(\oe))=\rho(-\eta(e))=\rho(\eta(e))^{-1}=\overline{\rho(\eta(e))}.
\]
In complete analogy to Equation~\eqref{eq:threetermzeta}, we have the following result.

\begin{theorem} [Theorem 18.15 in~\cite{2010Terras}] Let $p:\tX\to X$ be a free $G$-cover with abelian Galois group $G$ and let $\rho:G\to \mathbb{C}^*$ be a character of $G$. The Artin--Ihara $L$-function of $X$ satisfies
\begin{equation}
L(s,X,\rho)^{-1}=(1-s^2)^{g(X)-1}\det[I_n-sA_\rho+s^2(Q-I_n)].
\label{eq:threetermL}
\end{equation}
\label{thm:ArtinL}
\end{theorem}

Finally, we have the following relationship between the $L$-functions of a free $G$-cover $p:\tX\to X$ and the zeta function of $\tX$. First, it is clear that the $L$-functions valued at the trivial character are the corresponding zeta functions of $X$. Now let $\zeta_E(w_{\te\tf},\tX)$ be the edge zeta function on $\tX$, whose variables $w_{\te\tf}$ are indexed by pairs of oriented edges $\te,\tf$ of $\tX$ such that $t(\te)=s(\tf)$ and $\tf\neq \overline{\te}$. The \emph{$X$-specialization} $\zeta_E(w_{ef},\tX)$ of $\zeta_E(w_{\te\tf},\tX)$ is obtained by setting $w_{\te\tf}=w_{ef}$,
where $e=p(\te)$ and $f=p(\tf)$ are the corresponding edges of $X$. Similarly, the \emph{$X$-specialization} $\zeta_M(s,x_e,\tX)$ of the metric zeta function $\zeta_M(s,x_{\widetilde{e}},\tX)$ of $\tX$ is defined by setting $x_{\widetilde{e}}=x_e$, where $e=p(\widetilde{e})$, and there is no need to specialize the Ihara zeta function of $\tX$. 

\begin{theorem}[Corollary 18.11 and Theorem 19.3 in~\cite{2010Terras}] Let $p:\tX\to X$ be a free $G$-cover with abelian Galois group $G$, and let $\widehat{G}'$ be the set of nontrivial characters of $G$. The $X$-specialized edge and metric zeta functions of $\tX$ and the Ihara zeta function of $\tX$ are expressed in terms of the corresponding $L$-functions of $p:\tX\to X$ as follows:
\begin{align}
\zeta_E(w_{ef},\tX)&=\zeta_E(w_{ef},X)\prod_{\rho\in \widehat{G}'}L_E(w_{ef},\tX/X,\rho),\\
\zeta_M(s,x_e,\tX)&=\zeta_M(s,x_e,X)\prod_{\rho\in \widehat{G}'}L_M(s,x_e,\tX/X,\rho),\label{eq:metriczetafactorization}\\
\zeta(s,\tX)&=\zeta(s,X)\prod_{\rho\in \widehat{G}'}L(s,\tX/X,\rho).
\end{align}
\end{theorem}

\subsection{The zeta function and the Jacobian} We now express the Jacobian polynomial of a graph in terms of its metric zeta function. First, we recall a theorem of Northshield~\cite{1998Northshield} computing the order of the Jacobian in terms of the Ihara zeta function.

\begin{proposition} The reciprocal of the Ihara zeta function of a graph $X$ of genus $g$ has the following Taylor expansion at $s=1$:
\begin{equation}
\zeta(s,X)^{-1}=2^{g}(-1)^{g+1}(g-1)|\Jac(X)|(s-1)^g+O\left((s-1)^{g+1}\right).
\label{eq:s=1zeta}
\end{equation}
\label{prop:metricclassnumberformula}
\end{proposition}
In particular, if $g\geq 2$ then $\zeta(s,X)^{-1}$ has a zero of order $g$ at $s=1$, and the leading order term computes $|\Jac(X)|$. We now upgrade Northshield's result to compute the Jacobian polynomial in terms of the metric zeta function, by sneaking in the edge lengths $x_e$. 

\begin{proposition} The reciprocal of the metric zeta function of a graph $X$ of genus $g$ has the following Taylor expansion at $s=1$:
\begin{equation}
\zeta_M(s,x_e,X)^{-1}=2^{g}(-1)^{g+1}(g-1)J_X(x_e)(s-1)^g+O\left((s-1)^{g+1}\right).
\label{eq:s=1metriczeta}
\end{equation}
\label{prop:s=1zeta}
\end{proposition}
\begin{proof} The two-term determinant formula~\eqref{eq:twotermzeta} and the specialization~\eqref{eq:specializationzeta} imply that $\zeta_M(s,x_e,X)^{-1}$ is a polynomial in the quantities $s^{x_e}$. We need to compute the derivatives
\[
Z_k(x_e)=\left.\frac{\partial^k(\zeta_M(s,x_e,X)^{-1})}{\partial s^k}\right|_{s=1}
\]
for $k=0,\ldots,g$. Using the power rule, we see that $Z_k(x_e)$ is a degree $k$ polynomial in the $x_e$, which we can determine by finding its values for all integer inputs. 

As in~\eqref{eq:Jn}, let $\un=\{n_e\}_{e\in E(X)}$ be positive integers, and let $X_{\un}$ be the graph obtained from $X$ by replacing each edge $e\in E(X)$ with a chain of $n_e$ edges. There is a one-to-one correspondence between primes on $X$ and $X_{\un}$, but the path lengths on the latter are counted using the $n_e$. In other words, we compute the Ihara zeta function of $X_{\un}$ by substituting $x_e=n_e$ into the metric zeta function of $X$:
\[
\zeta(x,X_{\un})=\zeta_M(s,n_e,X).
\]
Equation~\eqref{eq:s=1metriczeta} then implies that for $k=0,\ldots,g-1$ we have
\[
Z_k(n_e)=\left.\frac{\partial^k(\zeta_M(s,n_e,X)^{-1})}{\partial s^k}\right|_{s=1}=
\left.\frac{\partial^k(\zeta(x,X_{\un})^{-1})}{\partial s^k}\right|_{s=1}=0.
\]
Since the $n_e$ are arbitrary positive integers, it follows that $Z_k(x_e)=0$ for $k=0,\ldots,g-1$. Similarly, using~\eqref{eq:s=1metriczeta} and~\eqref{eq:Jn} we see that
\[
Z_g(n_e)=\left.\frac{\partial^g(\zeta_M(s,n_e,X)^{-1})}{\partial s^k}\right|_{s=1}=
\left.\frac{\partial^g(\zeta(x,X_{\un})^{-1})}{\partial s^k}\right|_{s=1}=
g!2^g(-1)^{g+1}(g-1)J_X(n_e),
\]
therefore the left and right hand sides of this equality remain equal if the integers $n_e$ are replaced by the variables $x_e$. This completes the proof. 
\end{proof}

\section{The matroid of a Galois harmonic cover} \label{sec:matroid}

We now introduce our last technical tool, namely matroids. We do not review them here and refer the reader to a standard reference such as~\cite{2011Oxley}, and in any case we only use the elementary definitions of bases, independent sets, contraction, and duality. 

The starting point is the following observation. The Jacobian polynomial $J_X$ of a graph $X$ is a sum~\eqref{eq:JX} over the spanning trees $T$ of $X$. Hence $J_X$ is determined by the graphic matroid $\calM(X)$, whose bases are the spanning trees, or perhaps more accurately by the dual cographic matroid $\calM^*(X)$, since in~\eqref{eq:JX} we take products of the complementary edge lengths to spanning trees. Given a harmonic $G$-cover $p:\tX\to X$, the order of $\Jac(\tX)$ divides that of $\Jac(X)$ by Lagrange's theorem, therefore we may expect the Jacobian polynomial $J_{\tX}$ to be divisible by $J_X$. We show that this is indeed the case, and that the factors of the quotient are likewise expressible in terms of a matroid.

In~\cite{1982Zaslavsky}, Zaslavsky associated a matroid $\calM(X,\sigma)$ to a \emph{signed graph} $(X,\sigma)$, where $\sigma:E(X)\to \{\pm 1\}$ is a distribution of positive or negative signs on the edges of $X$. He then generalized this construction (see~\cite{1989Zaslavsky} and~\cite{1991Zaslavsky}) to associate a \emph{bias matroid} $\calM(X,\eta)$ to a \emph{gain graph}, which, in our terminology, is a graph $X$ together with a fixed $G$-voltage assignment $[\eta]\in H^1(X,G)$ (the case of signed graphs corresponds to $G=\ZZ/2\ZZ$). To define the contractions of $\calM(X,\eta)$ in Zaslavsky's framework, it is necessary to introduce additional objects called \emph{half edges} and \emph{loose edges}. In this section, we give an elementary interpretation of Zaslavsky's matroid in terms of the free $G$-cover $p:\tX\to X$ associated to the pair $(X,\eta)$. We then show that this interpretation works for dilated $G$-covers as well, and obviates the need to introduce new edge types. As in the rest of the paper, we assume $G$ to be abelian, but our results readily generalize to non-abelian groups.

\begin{definition} Let $p:\tX\to X$ be a harmonic $G$-cover. A subgraph $Y\subseteq X$ is called \emph{$G$-nontrivial} if the restricted cover $p|_{p^{-1}(Y)}:p^{-1}(Y)\to Y$ is not isomorphic to the trivial $G$-cover.
    
\end{definition}

We now imagine removing edges from $X$ one by one. This may split $X$ into connected components, and eventually (unless each vertex of $X$ is dilated) we will obtain a connected component that is $G$-trivial. We consider all ways of removing edges from $X$ such that this does not happen.

\begin{proposition} Let $p:\tX\to X$ be a nontrivial harmonic $G$-cover. The subsets
\begin{equation}
\calI(\calM^*(\tX/X))=\{F\subseteq E(X):\mbox{each connected component of }X\backslash F\mbox{ is $G$-nontrivial}\}
\label{eq:independent}
\end{equation}
are the independent sets of a matroid $\calM^*(\tX/X)$ with ground set $E(X)$ and rank
\begin{equation}
r(\calM^*(\tX/X))=g(X)-1+\left|\left\{v\in V(X):D(v)\neq 0\right\}\right|.
\label{eq:matroidrank}
\end{equation}
\label{prop:matroid}
\end{proposition}

\begin{proof} Since we are primarily interested in bases, we use $\calM^*(\tX/X)$ and its dual $\calM(\tX/X)$ interchangeably. We refer to $\calM(\tX/X)$ as the \emph{matroid of the $G$-cover} $p:\tX\to X$. We prove that $\calM^*(\tX/X)$ is a matroid by identifying its dual $\calM(\tX/X)$ with Zaslavsky's matroid $\calM(X,\eta)$ in the case when $p$ is free, and with its contractions for arbitrary $p$. 

We first give an explicit description of the bases of $\calM^*(\tX/X)$, in other words the maximal elements of $\calI(\calM^*(\tX/X))$. For this, we need to identify the minimal $G$-nontrivial subgraphs $Y\subseteq X$. Let $[\eta]\in H^1(X,D)$ be the dilated $G$-voltage assignment defining $p:\tX\to X$. If $Y$ contains a dilated vertex, then $Y$ is $G$-nontrivial by definition. If $Y$ has no dilated vertices, then the restricted cover $p|_{p^{-1}(Y)}:p^{-1}(Y)\to Y$ is free and is defined by the $G$-voltage assignment $[\eta|_Y]\in H^1(Y,G)$. The restricted cover $p|_{p^{-1}(Y)}$ is not isomorphic to the trivial cover if and only if $[\eta|_Y]\neq 0$, which happens if and only if there exists a cycle $C\subseteq Y$ such that the oriented sum of the $\eta_e$ along $e\in E(C)$ (the Frobenius element of $C$ in the terminology of Section~\ref{sec:zeta}) is not the identity.

Now let $Y\subseteq X$ be $G$-nontrivial. If $Y$ has two dilated vertices, then we can remove an edge on any path connecting them. If $Y$ has one dilated vertex and $g(Y)\geq 1$, then we can remove an edge from $Y$ without disconnecting it. If $Y$ has no dilation but $g(Y)\geq 2$, then we can remove an edge without disconnecting $Y$, while also preserving a cycle with nontrivial voltage. In all cases, the resulting connected components are also $G$-nontrivial. Hence we have proved the following lemma.

\begin{lemma} Let $F\subseteq E(X)$ and let $X\backslash F=X_1\sqcup\cdots\sqcup X_k $ be the decomposition into connected components. Then $F$ is a basis of $\calM^*(\tX/X)$ if and only if each $X_i$ is one of the following:
\begin{enumerate}

    \item A graph of genus one having no dilated vertices and such that the sum $\eta(X_i)$ of the voltages along the unique cycle of $X_i$ is not the identity. 
    \item A tree containing a unique dilated vertex.

\end{enumerate}

\label{lem:basis}
    
\end{lemma}

If $p:\tX\to X$ is a free $G$-cover, then there are no dilated vertices, so a basis of $\calM(\tX,X)$ (in other words, the complement of a basis of $\calM^*(\tX/X)$) is a union of genus one subgraphs of $X$, spanning all vertices and having nontrivial voltages along their unique circuits. By Theorem 2.1(g) in~\cite{1989Zaslavsky}, this is one of the cryptomorphic definitions of Zaslavsky's bias matroid, which proves that~\eqref{eq:independent} indeed defines a matroid on $E(X)$.

We now consider an arbitrary harmonic $G$-cover $p:\tX\to X$. By Lemma~\ref{lem:contraction}, there exists a free $G$-cover $p_f:\tX_f\to X_f$ such that $p$ is obtained from $p_f$ by contracting a set of edges $F_f\subset E(X_f)$. We show that $\calM^*(\tX/X)$ is a matroid by showing that it is obtained by deleting the set $F_f\subset E(X_f)$ from the matroid $\calM^*(\tX_f/X_f)$ (so that $\calM(\tX/X)$ is obtained from $\calM(\tX_f/X_f)$ by contracting $F_f$):
\[
\calM^*(\tX/X)=\calM^*(\tX_f/X_f)\backslash F_f,\quad \calM(\tX/X)=\calM(\tX_f/X_f)/F_f=(\calM^*(\tX_f/X_f)\backslash F_f)^*.
\]
Fix an independent set $F\in\calI(\calM^*(\tX_f/X_f)\backslash F_f)$, so that $F\in \calM^*(\tX_f/X_f)$ and $F\subseteq E(X_f)\backslash F_f$. Let $X_f\backslash F=(X_f)_1\sqcup \cdots\sqcup (X_f)_k$ be the connected component decomposition of the complement of $F$. The contracted edges lie on the $(X_f)_i$, therefore $X\backslash F=X_1\sqcup \cdots\sqcup X_k$, where $X_i$ is the contraction of $(X_f)_i$ along $F_f\cap E((X_f)_i)$. Whether or not a $G$-cover is trivial does change under contraction. Each $(X_f)_i$ is $G$-nontrivial, hence so is each $X_i$ and therefore $F\in \calI(\calM^*(\tX/X))$. Running this argument in reverse, we see that
\[
\calI(\calM^*(\tX/X))=\calI(\calM^*(\tX_f/X_f)\backslash F_f),
\]
hence $\calM^*(\tX/X)$ is indeed a matroid. 

We now prove formula~\eqref{eq:matroidrank} for the rank of $\calM^*(\tX/X)$ (which can also be found in Zaslavsky's work). If $p:\tX\to X$ is free, then $r(\calM^*(\tX/X))=g(X)-1$ follows from the following elementary counting argument, which we will use later and whose proof we omit:

\begin{lemma} Let $X$ be a graph of genus $g$, let $F\subseteq E(X)$ be a set of edges, and let $X\backslash F=X_1\sqcup\cdots\sqcup X_k$.
\begin{enumerate}
    \item If $|F|=g-1$, then either $g(X_i)=1$ for all $i$ or $g(X_i)=0$ for some $i$.
    \item Conversely, if $g(X_i)=1$ for all $i$, then $|F|=g-1$.
\end{enumerate}
\label{lem:counting}
\end{lemma}
We now use free $G$-covers as the base of an induction argument, and consider what happens to the rank of $\calM^*(\tX/X)$ when we contract a single edge $e\in E(X)$ to obtain the cover $p_e:\tX_e\to X_e$. There are several cases to consider.

If $e$ has distinct root vertices $u$ and $v$, then $g(X_e)=g(X)$. If $D(u)=0$ or $D(v)=0$ (or both), then $X_e$ has the same number of dilated vertices as $X$, and there exists a minimal $G$-nontrivial subgraph of $X$ containing $e$. Therefore $e$ does not lie in every basis of $\calM^*(\tX/X)$, hence $r(\calM^*(\tX/X))=r(\calM^*(\tX_e/X_e))$, and no terms in~\eqref{eq:matroidrank} change. On the other hand, if $D(u)\neq 0$ and $D(v)\neq 0$, then $X_e$ has one fewer dilated vertex than $X$, while $e$ lies in every basis of $\calM^*(\tX/X)$ (because $\{u,v,e\}$ is not minimally $G$-nontrivial), so that $r(\calM^*(\tX_e/X_e))=r(\calM^*(\tX/X))-1$. Hence both sides of~\eqref{eq:matroidrank} decrease by one.

The case when $e$ is a loop with root vertex $v$ is similar. The genus $g(X_e)=g(X)-1$ goes down by one. If $D(v)=0$ and $\eta(e)=0$, then $X$ and $X_e$ have the same number of dilated vertices. No subgraph containing $\{v,e\}$ is minimally $G$-nontrivial, hence $e$ must lie in every basis of $\calM^*(\tX/X)$. Therefore $r(\calM^*(\tX_e/X_e))=r(\calM^*(\tX/X))-1$, so that both sides of~\eqref{eq:matroidrank} decrease by one. If $D(v)=0$ and $\eta(e)\neq 0$, then $v$ is a dilated vertex in $X_e$, and $r(\calM^*(\tX_e/X_e))=r(\calM^*(\tX/X))$ since $\{v,e\}$ is a minimally $G$-nontrivial subgraph. In this case, both sides of~\eqref{eq:matroidrank} remain the same. Finally, if $v$ is dilated, then the number of dilated vertices does not change, but the rank decreases by one since $e$ lies in every basis of $\calM^*(\tX/X)$, so both sides of~\eqref{eq:matroidrank} decrease by one. This completes the proof.
\end{proof}

We now assume that the $G$-cover $p:\tX\to X$ is connected (hence nontrivial), and modify the definition of the matroid $\calM^*(\tX/X)$ by twisting the voltages by a nontrivial character $\rho$ of the Galois group $G$. Specifically, we replace $G$ and the dilation subgroups $D(v)$ by their images $\rho(G)$ and $\rho(D(v))$ in $\CC^*$, and the dilated $G$-voltage assignment $[\eta]\in H^1(X,D)$ by the dilated $\rho(G)$-voltage assignment $[\rho(\eta)]\in H^1(X,\rho(D))$. The result is the \emph{$\rho$-twisted matroid of the $G$-cover} $p:\tX\to X$.

\begin{proposition} Let $p:\tX\to X$ be a connected harmonic $G$-cover and let $\rho:G\to \mathbb{C}^*$ be a nontrivial character of $G$. There exists a matroid $\calM^*(\tX/X,\rho)$ with ground set $E(X)$ and rank
\[
r(\calM^*(\tX/X,\rho))=g(X)-1+\left|\left\{v\in V(X):\rho(D(v))\neq \{1\}\right\}\right|,
\]
whose bases are characterized by the following property. Let $F\subseteq E(X)$ be a set of edges, and let $X\backslash F=X_1\sqcup\cdots\sqcup X_k$ be the connected components of the complement. Then $F\in \calB(\calM^*(\tX/X,\rho))$ if and only if each $X_i$ is exactly one of the following:
\begin{enumerate}
    \item A graph of genus one such that $\rho(D(v))=\{1\}$ for all $v\in V(X_i)$, and such that $\rho(\eta(X_i))\neq 1$, where $\rho(\eta(X_i))$ is the oriented product of $\rho(\eta_e)$ along the unique cycle of $X_i$.
    \item A tree containing a unique vertex $v\in V(X_i)$ such that $\rho(D(v))\neq \{1\}$.
\end{enumerate}
\label{prop:twistedbases}
\end{proposition}

\begin{proof} This follows directly from Proposition~\ref{prop:matroid} and Lemma~\ref{lem:basis}, provided that we show that the $\rho(G)$-cover of $X$ determined by the $\rho(G)$-voltage assignment $[\rho(\eta)]\in H^1(X,\rho(D))$ is not the trivial $\rho(G)$-cover. If $\rho(D(v))\neq \{1\}$ for some $v\in v(X)$, then the $\rho(G)$-cover associated to $[\rho(\eta)]$ is not free and hence not trivial. Now suppose that $\rho(D(v))=\{1\}$ for all $v\in v(X)$ and $[\rho(\eta)]=0$ in $H^1(X,\rho(D))$. This implies that we can choose a representative $\eta$ for the original $G$-voltage assignment $[\eta]\in H^1(X,D)$ such that $\eta_e\in \Ker \rho$ for all $e\in E(X)$. The local description~\eqref{eq:dilatedvertexfiber}-\eqref{eq:dilatedtargetfiber} then shows that the subgraph of $\tX$ consisting of the vertices $\tv_{[g]}$ and the edges $\te_{g}$ indexed by $g\in \Ker \rho$ is not connected to the rest of $\tX$. This contradicts the assumption that $\tX$ is connected, hence $[\rho(\eta)]\in H^1(X,\rho(D))$ determines a nontrivial $\rho(G)$-cover. The proof now follows from Proposition~\ref{prop:matroid}, by replacing $G$ with $\rho(G)$ everywhere.

\end{proof}

\section{The main result} \label{sec:main}

We are now ready to state and prove the main result of the paper. Let $p:\tX\to X$ be a harmonic $G$-cover with abelian Galois group $G$. Recall that the Jacobian polynomial $J_X(x_e)$ is defined as a sum~\eqref{eq:JX} over the spanning trees of $X$, in other words over the bases of the cographic matroid $\calM^*(X)$. We now define analogous sums over the bases of the matroids $\calM^*(\tX/X,\rho)$, weighted in terms of $\rho$ and the dilated $G$-voltage assignment $[\eta]\in H^1(X,D)$ defining the cover.

\begin{definition} Let $p:\tX\to X$ be a harmonic $G$-cover with abelian Galois group $G$ defined by a dilated $G$-voltage assignment $[\eta]\in H^1(X,D)$, and let $\rho:G\to \mathbb{C}^*$ be a nontrivial character.
\begin{enumerate}
    \item Let $F\in \calB(\calM^*(\tX/X,\rho))$ be a basis of the matroid $\calM^*(\tX/X,\rho)$ and let $X\backslash F=X_1\sqcup\cdots\sqcup X_k$ be the connected components of the complement. We define the \emph{weight} $w_\rho(F)$ as
    \[
    w_\rho(F)=w_\rho(X_1)\cdots w_\rho(X_k),
    \]
    where the weights of the connected components $X_i$ (see Proposition~\ref{prop:twistedbases}) are defined as follows:
    \begin{enumerate}
        \item If $X_i$ is a graph of genus one such that $\rho(D(v))=\{1\}$ for all $v\in V(X_i)$, then
        \[
        w_{\rho}(X_i)=|1-\rho(\eta(X_i))|^2,
        \]
        where $\eta(X_i)$ is the Frobenius element (the sum of the oriented voltages $\eta_e$) of the unique cycle of $X_i$.
        \item If $X_i$ is a tree with a unique vertex $v\in V(X_i)$ such that $\rho(D(v))\subset \CC^*$ is not the trivial subgroup, then
        \[
        w_\rho(X_i)=1.
        \]
    \end{enumerate}
    \item The \emph{weight} $w(\calM^*(\tX/X,\rho))$ of the matroid $\calM^*(\tX/X,\rho)$ is
    \begin{equation}
    w(\calM^*(\tX/X,\rho))=\sum_{F\in \calB(\calM^*(\tX/X,\rho))}w_\rho(F).
    \label{eq:weight}
    \end{equation}
    \item The \emph{weight polynomial} $P_{\tX/X,\rho}(x_e)$ of the matroid $\calM^*(\tX/X,\rho)$ is the degree $r(\calM^*(\tX/X,\rho))$ polynomial in the edge lengths of $X$ given by
    \begin{equation}
    P_{\tX/X,\rho}(x_e)=\sum_{F\in \calB(\calM^*(\tX/X,\rho))}w_\rho(F)\prod_{e\in F}x_e.
    \label{eq:weightpolynomial}
    \end{equation}
\end{enumerate}
\label{def:weights}
\end{definition}

We note that for a connected component $X_i$ of $X\backslash F$ of the first type, the sum $\eta(X_i)$ depends on a choice of orientation up to sign, but the quantity $w_\rho(X_i)$ is well-defined. We also observe that the coefficients of $P_{\tX/X,\rho}(x_e)$, and hence the weights $w(\calM^*(\tX/X,\rho))$, lie in the real cyclotomic field, and are only guaranteed to be integers if $\rho$ is valued in the sixth roots of unity. Finally, we make the following observation in the case when $p:\tX\to X$ is free. If $X_0\subseteq X$ is a genus one subgraph with $\rho(\eta(X_0))=1$, then $w_\rho(X_0)=0$. Hence the sum in~\eqref{eq:weightpolynomial} can be taken over all bases of the matroid $\calM^*(\tX/X)$, and indeed over all $(g-1)$-element subsets $F\subseteq E(X)$ such that each connected component of $X\backslash F$ has genus one.

We are now ready to state our main result. Recall that the Jacobian polynomials $J_{\tX}(x_{\te})$ and $J_X(x_e)$ have variables indexed by the edges of $\tX$ and $X$, respectively. To compare the two, we define the \emph{$X$-specialized Jacobian polynomial} $J_{\tX}(x_e)$ of $\tX$ by setting $x_{\te}=x_{p(\te)}$ for all edges $\te\in E(\tX)$. We show that $J_X(x_e)$ divides $J_{\tX}(x_e)$ and compute the quotient in terms of the matroids $\calM^*(\tX/X,\rho)$.

\begin{theorem}
    Let $p:\tX\to X$ be a harmonic $G$-cover of graphs with abelian Galois group $G$ of order $N$. The $X$-specialized Jacobian polynomial of $\tX$ factors as
    \begin{equation}
    J_{\tX}(x_e)=\frac{1}{N}\prod_{v\in V(X)}|D(v)|^{N/|D(v)|}J_X(x_e)\prod_{\rho\in \widehat{G}'}P_{\tX/X,\rho}(x_e), \label{eq:mainformula1}
    \end{equation}
    where $D(v)$ is the dilation group at the vertex $v\in V(X)$, $J_X$ is the Jacobian polynomial of $X$, $\widehat{G}'$ is the set of nontrivial characters of $G$, and $P_{\tX/X,\rho}(x_e)$ is the weight polynomial of the matroid $\calM^*(\tX/X,\rho)$.
    \label{thm:main1}
\end{theorem}

Plugging $x_e=1$ for all $e\in E(X)$, we obtain our main enumerative result:

\begin{theorem}[Theorem~\ref{thm:mainintro}]
    Let $p:\tX\to X$ be a harmonic $G$-cover of graphs with abelian Galois group $G$ of order $N$. The number $|\Jac(\tX)|$ of spanning trees of $\tX$ is equal to 
    \begin{equation}
    |\Jac(\tX)|=\frac{1}{N}\prod_{v\in V(X)}|D(v)|^{N/|D(v)|}|\Jac(X)|\prod_{\rho\in \widehat{G}'}w(\calM^*(\tX/X,\rho)), \label{eq:enumerativeformula}
    \end{equation}
    where $D(v)$ is the dilation group at the vertex $v\in V(X)$, $|\Jac(X)|$ is the number of spanning trees of $X$, $\widehat{G}'$ is the set of nontrivial characters of $G$, and $w(\calM^*(\tX/X,\rho))$ is the weight of the matroid $\calM^*(\tX/X,\rho)$.
    \label{thm:main2}
\end{theorem}

Before we give the proof of Theorem~\ref{thm:main1}, we consider three examples: the icosahedral $\ZZ/5\ZZ$-quotient described in Examples~\ref{ex:icosahedron1} and~\ref{ex:icosahedron2}, a $\ZZ/6\ZZ$-cover for which the ranks of the matroids $\calM^*(\tX/X,\rho)$ depend on the $\rho$, and an example suggesting how Theorem~\ref{thm:main1} may be generalized to the non-abelian case.

\begin{example} \label{ex:icosahedron3} Consider the icosahedral quotient $p:\tX\to X$ shown on Figure~\ref{fig:icosahedron} and described in Examples~\ref{ex:icosahedron1} and~\ref{ex:icosahedron2}. The dilation groups are $D(v_1)=D(v_4)=\ZZ/5\ZZ$ and $D(v_2)=D(v_3)=1$. The Galois group $G=\ZZ/5\ZZ$ has four nontrivial characters
\[
\rho_j:\ZZ/5\ZZ\to \CC^*,\quad \rho_j(k)=e^{2\pi ijk/5},\quad j=1,\ldots,4.
\]
The matroid $\calM^*(\tX/X)$ and the four specializations $\calM^*(\tX/X,\rho_j)$ all have rank 4, and in fact have the same bases: the only four-element subsets of $E(X)$ that are not bases are $\{e_1,e_2,e_3,e_4\}$ and $\{e_3,e_4,e_5,e_6\}$. Let $f_j=|1-\rho_j(1)|^2$ for $j=1,\ldots,4$, so that
\[
f_1=f_4=\frac{5-\sqrt{5}}{2},\quad f_2=f_3=\frac{5+\sqrt{5}}{2}.
\]
The weights $w_j(F)=w_{\rho_j}(F)$ of the 13 bases $F\in \calB(\calM^*(\tX/X,\rho_j))$ can be evaluated directly from the definition and are given in the following table:
\begin{center}
\begin{tabular}{c|c}
   $F$  & $w_j(F)$ \\
   \hline
   $\{e_1,e_2,e_3,e_5\}$ & $1$ \\
   \hline
   $\{e_1,e_2,e_3,e_6\}$ & $f_j$ \\
   \hline
   $\{e_1,e_2,e_4,e_5\}$ & $1$ \\
   \hline
   $\{e_1,e_2,e_4,e_6\}$ & $f_j$ \\
   \hline
   $\{e_1,e_2,e_5,e_6\}$ & $f_j$ \\
   \hline
   $\{e_1,e_3,e_4,e_5\}$ & $f_j$ \\
   \hline
   $\{e_1,e_3,e_4,e_6\}$ & $f_j^2$ \\
   \hline
   $\{e_1,e_3,e_5,e_6\}$ & $f_j$ \\
   \hline
   $\{e_1,e_4,e_5,e_6\}$ & $f_j$ \\
   \hline
   $\{e_2,e_3,e_4,e_5\}$ & $1$ \\
   \hline
   $\{e_2,e_3,e_4,e_6\}$ & $f_j$ \\
   \hline
   $\{e_2,e_3,e_5,e_6\}$ & $1$ \\
   \hline
   $\{e_2,e_4,e_5,e_6\}$ & $1$ 
\end{tabular}
\end{center}
Therefore, the weights of the four matroids $\calM^*(\tX/X,\rho_j)$ are
\[
w(\calM^*(\tX/X,\rho_j))=5+7f_j+f_j^2=\begin{cases}
30-6\sqrt{5}, & j=1,4,\\
30+6\sqrt{5}, & j=2,3.
\end{cases}
\]
By Equation~\eqref{eq:enumerativeformula}, the number of spanning trees of the icosahedral graph is
\[
|\Jac(\tX)|=\frac{1}{5}\cdot 5^1\cdot 1^5\cdot 1^5\cdot 5^1\cdot 2\cdot (30-6\sqrt{5})^2\cdot(30+6\sqrt{5})^2=5184000.
\]

\end{example}

\begin{example} Consider the dilated $\ZZ/6\ZZ$-cover $p:\tX\to X$ of the dumbbell graph shown on Figure~\ref{fig:Z6}. The dilation groups are $D(v_1)=\ZZ/3\ZZ$ and $D(v_2)=\ZZ/2\ZZ$. The dilated cohomology group is 
\[
H^1(X,D)=\frac{\ZZ/6\ZZ}{D(v_1)}\oplus\frac{\ZZ/6\ZZ}{D(v_2)}=\ZZ/2\ZZ\oplus \ZZ/3\ZZ,
\]
where the two terms in the direct sum correspond to the loops $e_1$ and $e_2$. The dilated voltage assignment $[\eta]\in H^1(X,D)$ defining the cover is equal to $1$ on $e_1$ and $e_2$ (and is $0$ on $e_3$). 

Let $x_i$ denote the length of $e_i$ for $i=1,2,3$. The group $\ZZ/6\ZZ$ has five nontrivial characters
\[
\rho_j:\ZZ/5\ZZ\to \CC^*,\quad \rho_j(k)=e^{\pi ijk/3},\quad j=1,\ldots,5.
\]
For $j=1,5$ the images of $D(v_1)=\ZZ/3\ZZ$ and $D(v_2)=\ZZ/2\ZZ$ under $\rho_j$ are nontrivial. Hence the corresponding matroids $\calM^(\tX/X,\rho_j)$ have rank $3$, the only basis is the entire edge set $E(X)$, and the weight polynomial is simply
\[
P_{\tX/X,\rho_j}(x_e)=w_{\rho_j}(E(X))x_1x_2x_3=x_1x_2x_3,\quad j=1,5.
\]
For $j=2,4$ we have that $\rho_j(D(v_1))\neq\{1\}$ but $\rho_j(D(v_2))=\{1\}$, hence the rank is $r(\calM^*(\tX/X,\rho_j))=2$. The bases of $\calM^*(\tX/X,\rho_j)$ are $\{e_1,e_2\}$ and $\{e_1,e_3\}$, with weights
\[
w_{\rho_j}(\{e_1,e_2\}=1,\quad w_{\rho_j}(\{e_1,e_3\}=|1-\rho_j(1)|^2=3,\quad j=2,3.
\]
Hence the weight polynomials are
\[
P_{\tX/X,\rho_j}(x_e)=w_{\rho_j}(\{e_1,e_2\})x_1x_2+w_{\rho_j}(\{e_1,e_3\})x_1x_3=x_1x_2+3x_1x_3,\quad j=2,3.
\]
Conversely, $\rho_3(D(v_1))=\{1\}$ but $\rho_j(D(v_2))\neq\{1\}$, the rank is again $r(\calM^*(\tX/X,\rho_j))=2$, but the bases are $\{e_1,e_2\}$ and $\{e_2,e_3\}$ with weights 
\[
w_{\rho_3}(\{e_1,e_2\}=1,\quad w_{\rho_3}(\{e_2,e_3\}=|1-\rho_3(1)|^2=4.
\]
Hence the remaining weight polynomial is 
\[
P_{\tX/X,\rho_3}(x_e)=w_{\rho_3}(\{e_1,e_2\})x_1x_2+w_{\rho_3}(\{e_2,e_3\})x_2x_3=x_1x_2+4x_2x_3.
\]
Therefore, by Equation~\eqref{eq:mainformula1} the $X$-specialized Jacobian polynomial of $\tX$ factors as
\[
J_{\tX}(x_e)=\frac{1}{6}2^33^2x_1x_2(x_1x_2x_3)^2(x_1x_2+3x_1x_3)^2(x_1x_2+4x_2x_3)=12x_1^5x_2^4x_3^2(x_1+4x_3)(x_2+3x_3)^2,
\]
and the number of spanning trees of $\tX$ is equal to
\[
|\Jac(\tX)|=J_{\tX}(1)=960.
\]

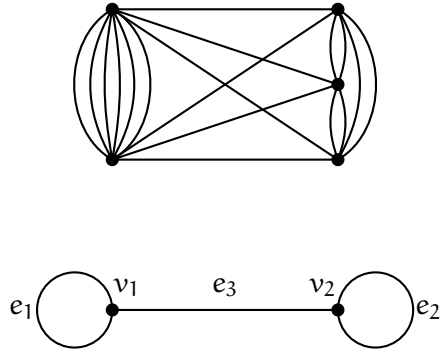
\begin{figure}
    \centering
\begin{tikzpicture}

\node at (0.2,0.3) {$v_1$};
\node at (2.8,0.3) {$v_2$};
\node at (-1.2,0) {$e_1$};
\node at (4.2,0) {$e_2$};
\node at (1.5,0.3) {$e_3$};

\draw[fill] (0,0) circle (.08);
\draw[fill] (3,0) circle (.08);

\draw[thick] (0,0) -- (3,0);
\draw[thick] (-0.5,0) circle (0.5);
\draw[thick] (3.5,0) circle (0.5);

\begin{scope}[shift={(0,2)}]

    \draw[fill] (0,0) circle (.08);
    \draw[fill] (0,2) circle (.08);
    \draw[fill] (3,0) circle (.08);
    \draw[fill] (3,1) circle (.08);
    \draw[fill] (3,2) circle (.08);

    \draw[thick] (0,0) -- (3,0);
    \draw[thick] (0,0) -- (3,1);
    \draw[thick] (0,0) -- (3,2);
    \draw[thick] (0,2) -- (3,0);
    \draw[thick] (0,2) -- (3,1);
    \draw[thick] (0,2) -- (3,2);

    \draw[bend left=10, thick] (0,0) to (0,2) ;
    \draw[bend left=30, thick] (0,0) to (0,2);
    \draw[bend left=60, thick] (0,0) to(0,2);
    \draw[bend right=10, thick] (0,0) to (0,2) ;
    \draw[bend right=30, thick] (0,0) to (0,2);
    \draw[bend right=60, thick] (0,0) to(0,2);

    \draw[bend right=20, thick] (3,0) to (3,1) ;
    \draw[bend left=20, thick] (3,0) to (3,1);
    
    \draw[bend right=30, thick] (3,0) to (3,2) ;
    \draw[bend right=60, thick] (3,0) to (3,2);
    
    \draw[bend right=20, thick] (3,1) to (3,2) ;
    \draw[bend left=20, thick] (3,1) to (3,2);

\end{scope}

\end{tikzpicture}

    \caption{Dilated $\ZZ/6\ZZ$-cover of the dumbbell graph}
    \label{fig:Z6}
\end{figure}

\end{example}

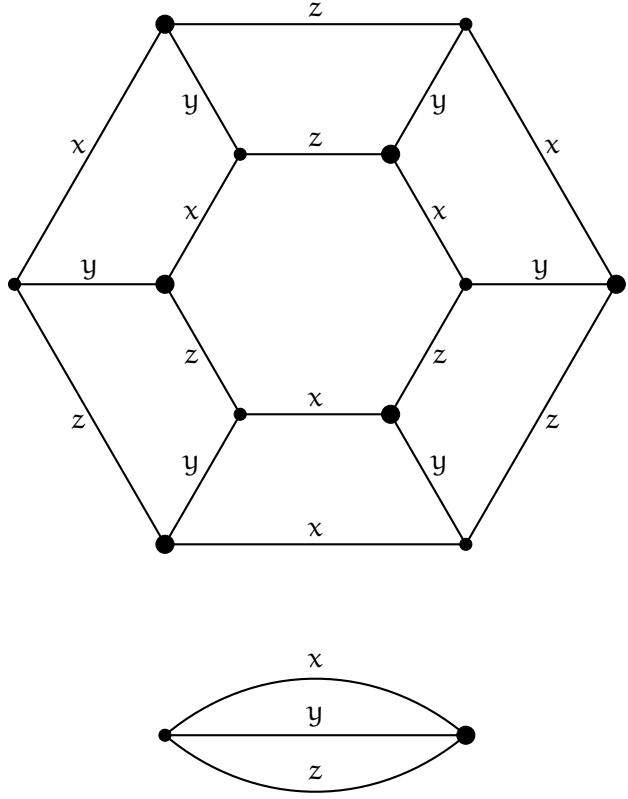
\begin{figure}
    \centering
    \begin{tikzpicture}

\node at (0,1) {$x$};
\node at (0,0.25) {$y$};
\node at (0,-0.5) {$z$};

\draw[fill] (-2,0) circle (.08);
\draw[fill] (2,0) circle (.12);

\draw[thick] (-2,0) -- (2,0);
\draw[bend left=40,thick] (-2,0) to (2,0);
\draw[bend right=40,thick] (-2,0) to (2,0);

\begin{scope}[shift={(0,6)}]

    \draw[fill] (2,0) circle (.08);
    \draw[fill] (1,1.73) circle (.12);
    \draw[fill] (-1,1.73) circle (.08);
    \draw[fill] (-2,0) circle (.12);
    \draw[fill] (1,-1.73) circle (.12);
    \draw[fill] (-1,-1.73) circle (.08);

    \draw[fill] (4,0) circle (.12);
    \draw[fill] (2,3.46) circle (.08);
    \draw[fill] (-2,3.46) circle (.12);
    \draw[fill] (-4,0) circle (.08);
    \draw[fill] (2,-3.46) circle (.08);
    \draw[fill] (-2,-3.46) circle (.12);
    
    \draw[thick] (2,0) -- (1,1.73) -- (-1,1.73) -- (-2,0) -- (-1,-1.73) -- (1,-1.73) -- (2,0);
    \draw[thick] (4,0) -- (2,3.46) -- (-2,3.46) -- (-4,0) -- (-2,-3.46) -- (2,-3.46) -- (4,0);

    \draw[thick] (2,0) -- (4,0);
    \draw[thick] (-2,0) -- (-4,0);
    \draw[thick] (1,1.73) -- (2,3.46);
    \draw[thick] (-1,1.73) -- (-2,3.46);
    \draw[thick] (1,-1.73) -- (2,-3.46);
    \draw[thick] (-1,-1.73) -- (-2,-3.46);


    \node at (1.65,0.96) {$x$};
    \node at (-1.65,0.96) {$x$};
    \node at (1.65,-0.96) {$z$};
    \node at (-1.65,-0.96) {$z$};
    \node at (0,1.93) {$z$};
    \node at (0,-1.53) {$x$};

    \node at (3.15,1.83) {$x$};
    \node at (3.15,-1.83) {$z$};
    \node at (-3.15,1.83) {$x$};
    \node at (-3.15,-1.83) {$z$};
    \node at (0,3.66) {$z$};
    \node at (0,-3.26) {$x$};

    \node at (3,0.2) {$y$};
    \node at (-3,0.2) {$y$};
    \node at (1.65,2.4) {$y$};
    \node at (1.65,-2.4) {$y$};
    \node at (-1.65,2.4) {$y$};
    \node at (-1.65,-2.4) {$y$};


\end{scope}

\end{tikzpicture}
    \caption{$S_3$-cover of the theta graph.}
    \label{fig:S3}
\end{figure}
\begin{example} Finally, we consider a free cover $p:\tX\to X$ with non-abelian Galois group $S_3$, shown on Figure~\ref{fig:S3}. The target graph $X$ is the theta graph with two vertices and three edges $e$, $f$, and $g$. The $S_3$-voltage (with respect to the left-to-right orientation) and the edge lengths are
\[
\eta(e)=(),\quad \eta(f)=(12),\quad \eta(g)=(123),\quad \ell(e)=x, \quad \ell(f)=y, \quad \ell(g)=z. 
\]
The source graph $\tX$ is a double hexagon. We distinguish one of the two vertices of $X$ and its preimages in $\tX$ by drawing them fatter, and label the edges of $X$ and $\tX$ by their lengths (so that the edges labeled $x$ in $\tX$ map to the edge labeled $x$ in $X$, and similarly for $y$ and $z$). A direct calculation shows that the $X$-specialized Jacobian polynomial of $\tX$ factors as
\begin{equation}
J_{\tX}(x,y,z)=6(xy+xz+yz)(x+z)(4xy+4xz+3y^2+4yz)^2,
\label{eq:JS3}
\end{equation}
and we analyze this product. By Theorem 19.3 in~\cite{2010Terras}, the $X$-specialized metric zeta function of $\tX$ factors as a product over the irreducible representations of $S_3$. Specifically,
\begin{equation}
\zeta_M(s,x_e,\tX)=\zeta_M(s,x_e,X)L_M(s,x_e,\tX/X,\sigma)L_M(s,x_e,\tX/X,\rho)^2,
\label{eq:zetaS3}
\end{equation}
where $\sigma$ and $\rho$ are the signature and standard representations, respectively. Since $J_{\tX}$ is obtained from $\zeta_M(s,x_e,\tX)$ at $s=1$, the factors in~\eqref{eq:JS3} correspond to the factors in~\eqref{eq:zetaS3}. Indeed, $xy+xz+yz$ is the Jacobian polynomial of $X$, while $x+z$ is the weight polynomial corresponding to $\sigma$ viewed as a character. Therefore, the last term is obtained from the $L$-function $L_M(s,x_e,\tX/X,\rho)$ at the rank two representation $\rho$, and it is natural to expect that it can be computed from the matroid $\calM(\tX/X)$ of the cover, twisted by the representation $\rho$. However, we do not know how to do so; in particular, the terms of the factor are no longer monomials, and one cannot expect a simple sum over the bases as in~\eqref{eq:weightpolynomial}.

\end{example}

We first prove Theorem~\ref{thm:main1} for free $G$-covers, by comparing the Taylor expansions of the zeta functions of $\tX$ and $X$ at $s=1$. We then use the resolution procedure of Lemma~\ref{lem:contraction} to derive the statement for arbitrary harmonic $G$-covers.

\begin{proposition} Let $p:\tX\to X$ be a free $G$-cover of a graph $X$ of genus $g$ with abelian Galois group $G$, and let $\rho:G\to \mathbb{C}^*$ be a nontrivial character. The reciprocal of the metric $L$-function has the following Taylor expansion at $s=1$:
\begin{equation}
L_M(s,x_e,\tX/X,\rho)^{-1}=2^{g-1}(-1)^{g-1} P_{\tX/X,\rho}(x_e)(s-1)^{g-1}+O\left((s-1)^g\right).
\label{eq:Lclass}
\end{equation}
\label{prop:Lclass}
\end{proposition}

\begin{proof} We use the same trick as in the proof of Proposition~\ref{prop:metricclassnumberformula}. We first show that the above formula holds for all $x_e=1$, when the left hand side is the reciprocal of the Artin--Ihara $L$-function $L(s,\tX/X,\rho)$. We then deduce that~\eqref{eq:Lclass} holds for all positive integer values of the $x_e$, by replacing the edges of $X$ with arbitrary chains of edges. Since the Taylor coefficients of $L_M(s,x_e,\tX/X,\rho)^{-1}$ are polynomials in the $x_e$, this proves the formula.

Equation~\eqref{eq:threetermL} implies that $L(s,\tX/X,\rho)^{-1}$ has a zero of order at least $g-1$ at $s=1$, and that the leading order term is
\begin{equation}
L(s,\tX/X,\rho)^{-1}=2^{g-1}(-1)^{g-1}\det L_{\rho}(s-1)^{g-1}+O\left((s-1)^g\right).
\label{eq:Lleading}
\end{equation}
Here $L_\rho=Q-A_{\rho}$ is the \emph{$\rho$-twisted Laplacian matrix} of the cover $p:\tX\to X$. Unlike the ordinary Laplacian, it is nonsingular.

Fix an orientation on $X$ and let $[\eta]\in H^1(X,G)$ be a $G$-voltage assignment defining $p:\tX\to X$. We factor the matrix $L_\rho$ as follows. Let $n=|V(G)|$ and $m=|E(G)|$, and introduce the $n\times m$ \emph{target} and \emph{$\rho$-twisted source}  matrices:
\[
T_{ve}=\begin{cases}
1, & t(e)=v, \\
0, & \mbox{otherwise,}
\end{cases}\quad
(S_\rho)_{ve}=\begin{cases}
\overline{\rho(\eta(e))}, & s(e)=v, \\
0, & \mbox{otherwise.}
\end{cases}
\]
It is elementary to verify that
\[
L_\rho=Q-A_{\rho}=B_{\rho}B_{\rho}^{\dagger},\quad B_{\rho}=T-S_{\rho}, 
\]
where $B_{\rho}$ is the \emph{$\rho$-twisted edge matrix} of $X$. Hence we can compute the determinant of $L_{\rho}$ using the Cauchy--Binet theorem:
\begin{equation}
\det L_{\rho}=\sum_{F\subseteq E(X)}\det B_{\rho}(X\backslash F) \det B_{\rho}(X\backslash F)^{\dagger}=\sum_{F\subseteq E(X)}\overline{\det B_{\rho}(X\backslash F)}^2.
\label{eq:BinetCauchy}
\end{equation}
Here the sum is taken over all subsets $F\subseteq E(X)$ having $m-n=g-1$ elements, and $B_{\rho}(X\backslash F)$ is the $\rho$-twisted edge matrix of $X\backslash F$.

Fix one such $F$, and let $X\backslash F=X_1\sqcup\cdots\sqcup X_k$ be the decomposition of the complement into connected components. The matrix $B_{\rho}(X\backslash F)$ decomposes into blocks $B_{\rho}(X_i)$ corresponding to the $X_i$. By Lemma~\ref{lem:counting}, either $g(X_i)=0$ for some $i$ or $g(X_i)=1$ for all $i$. In the former case the block $B_{\rho}(X_i)$ is not square (having one more vertex than edge), and therefore $\det B_{\rho}(X\backslash F)=0$. Hence in~\eqref{eq:BinetCauchy} it is enough to sum over the subsets $F\subseteq E(X)$ such that each connected component of $X\backslash F$ has genus one (again by Lemma~\ref{lem:counting}, any such subset necessarily has $g-1$ edges). In this case all blocks are square and
\begin{equation}
\overline{\det B_{\rho}(X\backslash F)}^2=\prod_{i=1}^k\overline{\det B_{\rho}(X_i)}^2=\prod_{i=1}^k\det L_\rho(X_i).
\label{eq:detBrho}
\end{equation}

Fix one of the components $X_i$, which consists of a unique cycle $C_i$ with trees attached. To compute the determinant $\det L_\rho(X_i)$, we choose a different orientation on $X_i$, namely one that follows the cycle and is oriented away from it along the trees. For any extremal edge $e\in E(X_i)$ with extremal vertex $v=t(e)$, the $v$-th row of $B_{\rho}(X_i)$ has a $1$ in the $e$-th column and zeroes everywhere else. Removing $v$ and $e$ does not change the determinant, and removing all extremal vertices we see that $\det B_\rho(X_i)=\det B_{\rho}(C_i)$.

We now cyclically label the vertices $V(C_i)=\{v_1,\ldots,v_l\}$ and edges $E(C_i)=\{e_1,\ldots,e_l\}$ so that $s(e_i)=v_i$ and $t(e_i)=e_{i+1}$. Denoting $\rho_i=\rho(\eta(e_i))$, the determinant can be easily computed by induction:
\[
\det B_{\rho}(C_i)=\left|\begin{matrix}
    -\overline{\rho_1} & 1 & 0 & \cdots & 0 \\
    0 & -\overline{\rho_2} & 1 & \cdots & 0 \\
    0 & 0 & -\overline{\rho_3} & \cdots & 0 \\
    \cdots & \cdots & \cdots & \cdots & \cdots \\
    1 & 0 & 0 & \cdots & -\overline{\rho_l}
\end{matrix}\right|=(-1)^{l+1}\left(1-\overline{\rho_1}\cdots\overline{\rho_l}\right)=(-1)^{l+1}(1-\overline{\rho(\eta(X_i))}).
\]
Therefore
\begin{equation}
\det L_\rho(X_i)=\overline{\det B_{\rho}(X_i)}^2=\overline{\det B_{\rho}(C_i)}^2= \left|1-\overline{\rho(\eta(X_i))}\right|^2=w_\rho(X_i).
\label{eq:detLrho}
\end{equation}
Plugging~\eqref{eq:detLrho} and~\eqref{eq:detBrho} into~\eqref{eq:BinetCauchy}, we see that
\begin{equation}
\det L_{\rho}=\sum_{F\in \calB(\calM^*(\tX/X,\rho))}w_\rho(F)=
\left.P_{\tX/X,\rho}(x_e)\right|_{x_e=1}.
\label{eq:detLrho2}
\end{equation}
The Artin--Ihara $L$-function $L(s,\tX/X,\rho)$ is the metric $L$-function specialized at $x_e=1$, so plugging~\eqref{eq:detLrho2} into~\eqref{eq:Lleading} we see that~\eqref{eq:Lclass} holds when all $x_e=1$.

We now fix positive integers $\un=\{n_e\}_{e\in E(X)}$ and construct the graph $X_{\un}$ by replacing each edge $e\in E(X)$ by a chain of $n_e$ edges (see proof of Proposition~\ref{prop:s=1zeta}). For each $e\in E(X)$, we arbitrarily pick one of these edges $e'$ and set $(\eta_{\un})_{e'}=\eta_e$, and we set the voltages on the remaining edges to be trivial. The resulting $G$-voltage assignment $[\eta_{\un}]\in H^1(X_{\un},G)$ defines a free $G$-cover $p_{\un}:\tX_{\un}\to X_{\un}$ of the same topological type as $p:\tX\to X$, but with each edge above and below extended into a chain. 

Let $F=\{e_1,\ldots,e_{g-1}\}\in \calB(\calM^*(\tX/X,\rho))$ be a basis of the dual matroid. For each $i=1,\ldots,g-1$, pick one of the $n_{e_i}$ edges in the chain in $X_{\un}$ corresponding to $e_i$. The result is a basis of $\calM^*(\tX_{\un}/X_{\un},\rho)$ having the same weight as $X$, and conversely every basis of $\calM^*(\tX_{\un}/X_{\un},\rho)$ is obtained in this way. In other words, each $F\in \calB(\calM^*(\tX/X,\rho))$ corresponds to $\prod_{e\in F}n_e$ bases of $\calM^*(\tX_{\un}/X_{\un},\rho)$, all having the same weight as $F$. 
Hence~\eqref{eq:detLrho2} implies that
\begin{equation}
\det L_{\rho}(X_{\un})=\sum_{F\in \calB(\calM^*(\tX_{\un}/X_{\un},\rho))}w_\rho(F)=
\sum_{F\in \calB(\calM^*(\tX/X,\rho))}w_\rho(F)\prod_{e\in F}n_e=
\left.P_{\tX/X,\rho}(x_e)\right|_{x_e=n_e}.
\label{eq:detLrhoun}
\end{equation}
Since $L(s,\tX_{\un}/X_{\un},\rho)$ is the specialization of $L_M(s,x_e,\tX/X,\rho)$ at $x_e=n_e$, plugging~\eqref{eq:detLrhoun} into~\eqref{eq:Lleading} shows that~\eqref{eq:Lclass} holds when $x_e=n_e$, and therefore for arbitrary $x_e$.

\end{proof}


\begin{proof}[Proof of Theorem~\ref{thm:main1}] We first assume that $f:\tX\to X$ is a free $G$-cover, so that $|D(v)|=1$ for all $v\in V(X)$. Let
\[
g=g(X)=|E(X)|-|V(X)|+1,\quad g(\tX)=N(g-1)+1
\]
be the genera of $X$ and $\tX$, respectively. The degrees of the weight polynomials $P_{\tX/X,\rho}(x_e)$ are all equal to $g-1$. By Propositions~\ref{prop:metricclassnumberformula} and~\ref{prop:Lclass}, the reciprocals of the $X$-specialized metric zeta function of $\tX$, the metric zeta function of $X$, and the metric $L$-functions at the nontrivial characters have the following leading order terms at $s=1:$
\begin{align*}
\zeta_M(s,x_e,\tX)^{-1}&=2^{N(g-1)+1}(-1)^{N(g-1)+2}N(g-1)J_{\tX}(x_e)(s-1)^{N(g-1)+1}+\cdots,\\
\zeta_M(s,x_e,X)^{-1}&=2^{g}(-1)^{g+1}(g-1)J_X(x_e)(s-1)^g+\cdots,\\
L_M(s,x_e,\tX/X,\rho)^{-1}&=2^{g-1}(-1)^{g-1} P_{\tX/X,\rho}(x_e)(s-1)^{g-1}+\cdots.
\end{align*}
On other other hand, $\zeta_M(s,x_e,\tX)$ factors~\eqref{eq:metriczetafactorization}  as a product of $\zeta_M(s,x_e,X)$ and the $L_M(s,x_e,\tX/X,\rho)$ at the $N-1$ nontrivial characters of $G$. Comparing the leading order terms, we obtain~\eqref{eq:mainformula1}.

By Lemma~\ref{lem:contraction}, an arbitrary harmonic $G$-cover can be obtained from a free $G$-cover by contracting a number of loops on the base of the latter. Hence, to finish the proof of Theorem~\ref{thm:main1}, it is sufficient to show that if formula~\eqref{eq:mainformula1} holds for a harmonic $G$-cover $f:\tX\to X$, then it also holds for the cover $f_0:\tX_0\to X_0$ obtained by contracting an arbitrary loop $e_0\in E(X)$.

Let $v_0\in V(X)$ be the root vertex of the loop $e_0$, let $x_0$ denote the length of $e_0$, let $D(v_0)\subseteq G$ be the dilation subgroup of $v_0$ on the graph $X$, and let $\eta_0\in G/D(v_0)$ be the dilated $G$-voltage on the loop $e_0$ with a choice of orientation. The dilation group of the vertex $v_0$ on the graph $X_0$ is $D(v_0)+\eta_0G$, where $\eta_0G$ is the subgroup generated by $\eta_0$. Let $H=|D(v_0)+\eta_0G|$ and $M=|D(v_0)|$, so that $H/M$ is the order of $\eta_0$ in $(D(v_0)+\eta_0G)/D(v_0)$. The fiber $p^{-1}(v_0)$ is identified with the quotient group $G/D(v_0)$, while $p^{-1}(e_0)=G$, and the source and target maps from $G=p^{-1}(e_0)$ to $G/D(v_0)=p^{-1}(v_0)$ are
\[
s:G\to G/D(v_0),\quad g\mapsto g+D(v_0),\quad t:G\to G/D(v_0),\quad g\mapsto g+\eta_0+D(v_0).
\]
Hence the preimage of the subgraph $\{v_0,e_0\}\subset X$ in $\tX$ consists of $N/H$ isomorphic graphs $C_1,\ldots,C_{N/H}$, each of which is a cyclic arrangement of $H/M$ vertices, with each adjacent pair of vertices joined by $M$ edges. Hence
\[
g(X)=g(X_0)+1,\quad g(\tX)=g(\tX_0)+N-\frac{N}{M}+\frac{N}{H}.
\]

We first determine the relationship between the $X$- and $X_0$-specialized Jacobian polynomials of $\tX$ and $\tX_0$, essentially by repeatedly applying Equation~\eqref{eq:Jedgecontraction}. There are two types of spanning trees $T\subset \tX$.

\begin{enumerate} \item If $T\cap C_i$ is a spanning tree of $C_i$ for each $i=1,\ldots,N/H$, then the contraction $T_0$ of $T$ is a spanning tree of $\tX_0$, and conversely every spanning tree $T_0$ of $\tX_0$ can be extended to a spanning tree $T$ of $X$ by gluing in a choice of spanning tree for each $C_i$. A spanning tree of $C_i$ is obtained by removing all edges between two adjacent vertices ($H/M$ choices for the pair), and all but one edge between each other pair of adjacent vertices ($M$ choices for each other pair). Hence the number of spanning trees of each $C_i$ is equal to
\[
|\Jac(C_i)|=\frac{H}{M}M^{H/M-1}=HM^{H/M-2}.
\]
The complementary edges $F=E(\tX)\backslash E(T)$ consist of the complementary edges $F_0=E(\tX_0)\backslash E(T_0)$ of the contracted tree $T_0\subset \tX_0$ and the $N-N/M+N/H$ preimages of $e_0$ (each having length $x_0$) that form the complements of the spanning trees $T\cap C_i\subset C_i$. 

\item On the other hand, if $T\cap C_i$ is not a spanning tree of $C_i$ for some $i=1,\ldots,N/H$, then the complement $F=E(\tX)\backslash E(T)$ contains at least  $N-N/M+N/H+1$ preimages of $e_0$. 

\end{enumerate}
Counting all choices, we see that
\begin{align}
J_{\tX}(x_e)=\sum_{T\subset X\,\mathrm{ type }\,1}\prod_{e\in E(X)\backslash E(T)}x_e+\sum_{T\subset X\,\mathrm{ type }\,2}\prod_{e\in E(X)\backslash E(T)}x_e
=\left[HM^{H/M-2}\right]^{N/H}x_0^{N-N/M+N/H}J_{\tX_0}(x_e)+\notag\\
+O\left(x_0^{N-N/M+N/H+1}\right)=H^{N/H}M^{N/M-2N/H}x_0^{N-N/M+N/H}J_{\tX_0}(x_e)+O\left(x_0^{N-N/M+N/H+1}\right).
\label{eq:Jtxcontact}
\end{align}

We now determine what happens to the right hand side of Equation~\eqref{eq:mainformula1} when we contract the loop $e_0$. By Lemma~\ref{lem:contraction} we have
\begin{equation}
J_X(x_e)=x_0J_{X_0}(x_e).
\label{eq:JXcontract}
\end{equation}
Now let $\rho:G\to \mathbb{C}^*$ be a nontrivial character. There are three possibilities.

\begin{enumerate}
    \item \label{item:type1} $\rho(D(v_0))=\{1\}$ and $\rho(\eta_0)=1$. The $e_0$ is a coloop of the matroid $\calM^*(\tX/X,\rho)$, and $F\mapsto F\backslash \{e_0\}$ is a bijection between $\calB(\calM^*(\tX/X,\rho))$ and $\calB(\calM^*(\tX_0/X_0,\rho))$. Furthermore, the weights of $F$ and $F\backslash \{e_0\}$ (in terms of the two different matroids) are equal:
    \[
    w_{X,\rho}(F)=w_{X_0,\rho}(F\backslash \{e_0\}).
    \]
    It follows that
    \begin{equation}
        P_{\tX/X,\rho}=\sum_{F\in \calB(\calM^*(\tX/X,\rho))}w_{X,\rho}(F)\prod_{e\in F}x_e=
        \sum_{F\backslash\{e_0\}\in \calB(\calM^*(\tX/X,\rho))}w_{X_0,\rho}(F\backslash\{e_0\})x_0\prod_{e\in F\backslash\{e_0\}}x_e=
        x_0P_{\tX_0/X_0,\rho}.
\label{eq:Pcontractedcoloop}
    \end{equation}
    \item \label{item:type2} $\rho(D(v_0))=\{1\}$ and $\rho(\eta_0)\neq 1$. Here $e_0$ is not a coloop of $\calM^*(\tX/X,\rho)$, the ranks of $\calM^*(\tX/X,\rho)$ and $\calM^*(\tX_0/X_0,\rho)$ are equal, and there is a bijection
\[
\left\{F\in \calB(\calM^*(\tX/X,\rho)):e_0\notin F\right\}\leftrightarrow \left\{F\in \calB(\calM^*(\tX_0/X_0,\rho))\right\}.
\]
Let $F\in \calB(\calM^*(\tX/X,\rho))$ be a basis not containing $e_0$, let $X\backslash F=X_1\sqcup \cdots\sqcup X_k$ be the connected component decomposition of the complement, and assume without loss of generality that the loop $e_0$ lies on $X_1$. Then $X_1$ is a graph of type (a) according to Definition~\ref{def:weights} and has weight $w_{X,\rho}(X_1)=|1-\rho(\eta_0)|^2$. Contracting the loop $e_0$, we obtain
\[
X_0\backslash F=X_{1,0}\sqcup X_2\sqcup \cdots\sqcup X_k,
\]
where $X_{1,0}$ is now a graph of type (b) and has weight $w_{X_0,\rho}(X_{1,0})=1$. The weights of the other connected components do not change, hence
\[
w_{X,\rho}(F)=|1-\rho(\eta_0)|^2w_{X_0,\rho}(F),
\]
and therefore
\[
P_{\tX/X,\rho}(x_e)=\sum_{F:e_0\notin F}w_{X,\rho}(F)\prod_{e\in F}x_e+\sum_{F:e_0\in F}w_{X,\rho}(F)\prod_{e\in F}x_e=
\]
\[
=\sum_{F\in \calB(\calM^*(\tX_0/X_0,\rho))}|1-\rho(\eta_0)|^2w_{X_0,\rho}(F)\prod_{e\in F}x_e+O(x_0)=|1-\rho(\eta_0)|^2P_{\tX_0/X_0,\rho}(x_e)+O(x_0).
\]

\item \label{item:type3} $\rho(D(v_0))\neq \{1\}$. Then $e_0$ is again a coloop of $\calM^*(\tX/X,\rho)$, and the same relation~\eqref{eq:Pcontractedcoloop} holds as for type~\ref{item:type1}.
\end{enumerate}

We now count the number of characters of each type. There is a natural bijection between nontrivial characters of type~\ref{item:type1} and nontrivial characters on the group $G/(D(v_0)+\eta_0G)$, hence there are $N/H-1$ characters of this type. Similarly, the characters vanishing on $D(v_0)$ are identified with the characters on the group $G/D(v_0)$. There are $N/M$ such characters, split into $N/H$ equivalence classes according to their values on the cyclic subgroup $(D(v_0)+\eta_0G)/D(v_0)$ generated by $\eta_0$. Hence there are
\[
\frac{H}{M}\left(\frac{N}{H}-1\right)=\frac{N}{M}-\frac{H}{M}
\]
characters of type~\ref{item:type2}. In addition, we calculate the product of $|1-\rho(\eta_0)|^2$ over all type~\ref{item:type2} characters:
\[
\prod_{\rho:G\to\mathbb{C}^*\,\mathrm{type}\,\ref{item:type2}}|1-\rho(\eta_0)|^2=\prod_{\rho:G/D(v)\to \mathbb{C}^*,\rho(\eta_0)\neq 0}|1-\rho(\eta_0)|^2=
\]
\[
=
\left[\prod_{\rho:(D(v_0)+\eta_0G)/D(v_0)\to \mathbb{C}^*,\rho(\eta_0)\neq 1}|1-\rho(\eta_0)|^2
\right]^{N/H}=\left[\left(\frac{H}{M}\right)^2\right]^{N/H}=H^{2N/H}M^{-2N/H}.
\]
In the last step, we use the following calculation. The group $(D(v_0)+\eta_0G)/D(v_0)$ is cyclic of order $H/M$ with generator $\eta_0$. Hence there are $H/M-1$ characters $\rho_1,\ldots,\rho_{H/M-1}$ in the product and $\rho_j(\eta_0)=\zeta^j$, where $\zeta=e^{2\pi i M/H}$. Therefore the product is given by
\[
\prod_{j=1}^{H/M-1}(1-\zeta^j)=\left.\prod_{j=1}^{H/M-1}(z-\zeta^j)\right|_{z=1}=\left.\frac{z^{H/M}-1}{z-1}\right|_{z=1}=\left.(z^{H/M-1}+\cdots+z+1)\right|_{z=1}=\frac{H}{M}.
\]
Finally, the remaining
\[
N-1-\left(\frac{N}{H}-1\right)-\left(\frac{N}{M}-\frac{H}{M}\right)N=N-\frac{N}{H}-\frac{N}{M}+\frac{H}{M}
\]
characters are of type~\ref{item:type3}.

We can therefore evaluate the product
\begin{align}
\prod_{\rho\in \widehat{G}'}P_{\tX/X,\rho}(x_e)=\prod_{\rho\,\mathrm{types}~\ref{item:type1}\,\mathrm{and}~\ref{item:type3}}x_0P_{\tX_0/X_0,\rho}(x_e)
\prod_{\rho\,\mathrm{type}~\ref{item:type2}}\left[|1-\rho(\eta_0)|^2P_{\tX_0/X_0,\rho}(x_e)+O(x_0)\right]=\notag\\
=x_0^{N-N/M+H/M-1}H^{2N/H}M^{-2N/H}\prod_{\rho\in \widehat{G}'}P_{\tX_0/X_0,\rho}(x_e)+
O\left(x_0^{N-N/M+H/M}\right).
\label{eq:Pcontact}
\end{align}

Now assume that the cover $p:\tX\to X$ satisfies Equation~\eqref{eq:mainformula1}, which we write as
\[
J_{\tX}(x_e)=\frac{1}{N}M^{N/M}\prod_{v\neq v_0}|D(v)|^{N/|D(v)|}J_X(x_e)\prod_{\rho\in \widehat{G}'}P_{\tX/X,\rho}(x_e).
\]
Plugging~\eqref{eq:Jtxcontact} in the left hand side,~\eqref{eq:JXcontract} and~\eqref{eq:Pcontact} in the right hand side, dividing by $x_0^{N-N/M+H/M}$, and then setting $x_0=0$ to kill higher order terms, we obtain
\[
J_{\tX_0}(x_e)=\frac{1}{N}H^{N/H}\prod_{v\neq v_0}|D(v)|^{N/|D(v)|}J_{X_0}(x_e)\prod_{\rho\in \widehat{G}'}P_{\tX_0/X_0,\rho}(x_e),
\]
which is exactly Equation~\eqref{eq:mainformula1} for the cover $p_0:\tX_0\to X_0$, having a dilation group of order $H=|D(v_0)+\eta_0G|$ at $v_0$. This completes the proof.


\end{proof}

\bibliographystyle{alpha}
\bibliography{references}

\end{document}